\documentclass[preprint,12pt]{elsarticle}
\usepackage[marginratio=1:1,height=9in,width=6.5in]{geometry}

\usepackage{amssymb}
\usepackage{amsmath}
\usepackage{amsthm}
\newtheorem{proposition}{Proposition}
\newtheorem{theorem}{Theorem}

\usepackage{algorithmic}
\usepackage{algorithm}

\usepackage{amsfonts}
\usepackage{xcolor}
\usepackage{hyperref}
\usepackage{graphicx}
\usepackage{amssymb}
\usepackage{bbm}
\usepackage{mathtools}
\usepackage{latexsym}
\usepackage{listings}
\usepackage{multirow}
\usepackage{enumitem}
\usepackage{siunitx}
\usepackage{booktabs}

\newcommand{\hm}{\widehat{\mu}}
\newcommand{\hD}{\widehat{D}}
\newcommand{\hS}{\widehat{\Sigma}}
\newcommand{\Sxx}{\Sigma^{xx}}
\newcommand{\Sxv}{\Sigma^{xv}}
\newcommand{\Svx}{\Sigma^{vx}}
\newcommand{\Svv}{\Sigma^{vv}}
\newcommand{\ve}{\varepsilon}

\newcommand{\EE}{\mathbb{E}}
\newcommand{\RR}{\mathbb{R}}

\newcommand{\Dh}{\mathrm{D}_H}
\newcommand{\DV}{\mathrm{D}_V}
\newcommand{\fvt}{f_{v\theta}}
\newcommand{\apdx}{Appendix}

\newcommand{\rd}{\mathrm{d}}

\newcommand{\pt}{\partial_t}

\newcommand{\nt}{\nabla_\theta}
\newcommand{\nv}{\nabla_v}
\newcommand{\nx}{\nabla_x}
\newcommand{\dt}{\Delta t}

\newcommand{\lam}{\lambda}

\newcommand{\wh}{\widehat}

\newcommand{\tp}{^{\top}}

\newcommand{\parentheses}[1]{\left(#1\right)}
\newcommand{\sqbra}[1]{\left[#1\right]}

\newcommand{\abs}[1]{\left|#1\right|}
\newcommand{\norm}[1]{\left\|#1\right\|}

\newcommand{\rom}[1]{\uppercase\expandafter{\romannumeral #1\relax}}
\newcommand{\KL}[2]{\mathrm{D}_{\mathrm{KL}}\left(#1 \,||\, #2\right)}

\DeclareMathOperator{\Tr}{Tr}

\journal{Journal of Computational Physics}

\begin{document}

\begin{frontmatter}



\title{Simulating Fokker--Planck equations via mean field control of score-based normalizing flows}

\author[labelUCLA]{Mo Zhou} 
\author[labelUCLA]{Stanley Osher}
\author[labelUSC]{Wuchen Li}

\affiliation[labelUCLA]{organization={Department of Mathematics},
            addressline={University of California}, 
            city={Los Angeles},
            postcode={90095}, 
            state={CA},
            country={USA}}

\affiliation[labelUSC]{organization={Department of Mathematics},
            addressline={University of South Carolina},
            city={Columbia},
            postcode={29208},
            state={SC},
            country={USA}}




\begin{abstract}
The Fokker--Planck (FP) equation governs the evolution of densities for stochastic dynamics of physical systems, such as the Langevin dynamics and the Lorenz system. This work simulates FP equations through a mean field control (MFC) problem. We first formulate the FP equation as a continuity equation, where the velocity field consists of the drift function and the score function, i.e., the gradient of the logarithm of the density function. Next, we design a MFC problem that matches the velocity fields in a continuity equation with the ones in the FP equation. The score functions along deterministic trajectories are computed efficiently through the score-based normalizing flow, which only relies on the derivatives of the parameterized velocity fields. Numerical results, including Langevin dynamics, underdamped Langevin dynamics, chaotic systems, and high-dimensional interacting particle systems validate the effectiveness and scalability of our proposed algorithm. A convergence analysis is conducted for our algorithm on the FP equation of Ornstein--Uhlenbeck processes.
\end{abstract}




\end{frontmatter}



\section{Introduction}
Simulating complex physical dynamics, especially chaotic systems, is a longstanding challenge \cite{yang2023optimal}. Important examples include Langevin dynamics and the Lorenz system. These systems have attracted significant attention across statistical physics, stochastic control, and machine learning disciplines. The Fokker--Planck (FP) equation, which describes the evolution of probability densities, is studied to analyze these systems because it captures essential properties, such as the free energy dissipation  \cite{bertini2015macroscopic,amari2016information,li2020fisher,te2020classical,lee2021controlling}. The free energy functional serves as a Lyapunov function for FP dynamics and characterizes convergence toward equilibrium through its dissipation \cite{jordan1998variational,ambrosio2008gradient}. In statistical physics, this corresponds to entropy production and the second law of thermodynamics. Numerical schemes that fail to preserve this dissipation structure may produce incorrect long-time behavior. Therefore, designing fast, efficient, and accurate algorithms for FP equations that preserve the free energy dissipation is a central problem in the scientific computing and mathematical data science communities.

An algorithm for simulating the FP equation can be designed from a mean field control (MFC) problem. The MFC models the optimal control of interacting particles' trajectories by minimizing a cost functional, which has wide applications in finance \cite{carmona2021deep}, epidemic control \cite{lee2021controlling}, and statistical physics \cite{bertini2015macroscopic}. In particular, the MFC for simulating the FP equation can be viewed as minimizing a discrepancy loss to learn the system’s dynamics, which is also known as the flow matching problem \cite{li2023self,albergo2023stochastic,lindsey2025mne}. Here, the diffusion in the FP equation can be represented using the score function, which is the gradient of the logarithm of the probability density function. Using the score function, the FP equation can be reformulated as a continuity equation for a deterministic dynamic. As a result, one needs to solve an MFC problem involving the score function, which matches the velocity fields for the continuity equation. In practice, efficiently computing the score function along trajectories remains a computational challenge.

In recent years, the machine learning communities have introduced normalizing flows \cite{rezende2015variational} and neural ordinary differential equations (ODEs) \cite{chen2018neural} for modeling and approximating deterministic dynamical systems. Normalizing flows construct a sequence of invertible push-forward maps that sequentially transfer one probability distribution into another, enabling flexible density estimation. Neural ODEs interpret the limit of residual neural networks as continuous-time dynamical systems. Despite their empirical success, these approaches' scalability and convergence behaviors are still not fully understood, particularly in stochastic systems, such as complex physical and chaotic systems. Recently, new methods have been proposed to estimate score functions via deterministic ODEs governed by learned velocity fields \cite{shen2022self, zhou2025score}, which is known as the score-based normalizing flow. These models bridge stochastic dynamics with deterministic control through score-based transformations. {\em A natural question arises: Can we apply the score-based neural ODE to simulate the general FP equations of Langevin dynamics?}

This work proposes an MFC framework for simulating FP equations using score-based normalizing flows. We formulate flow matching for the FP equation as a MFC problem, where the stochastic dynamics are converted into deterministic ones using the score function. We introduce score-based neural ODEs and score-based normalizing flows as efficient tools for approximating score functions and modeling the evolution of probability densities. The deterministic velocity field is parameterized via neural networks and optimized by minimizing a variational discrepancy loss. Our framework enables accurate computation of key quantities such as the free energy and its dissipation. Numerical experiments on Langevin dynamics, underdamped Langevin dynamics (ULDs), and various chaotic systems validate the effectiveness of our proposed algorithms. These experiments confirm that our approach accurately captures the density evolution and reliably computes free energy and its dissipation over time. We provide a convergence analysis for the Ornstein–-Uhlenbeck (OU) process.

\subsection*{Related work} 
The FP equation with gradient drift can be interpreted as a gradient flow of the relative entropy or free energy functional in the space of probability density with Wasserstein-$2$ metric \cite{ambrosio2008gradient, jordan1998variational}. This gradient flow implies the dissipation property of free energy. For FP equations with non-gradient drifts, the free energies are still dissipative. In statistical physics, the MFC formulation is introduced to study the FP equation \cite{bertini2015macroscopic}.  Mathematically, the theory of MFC and its counterpart, mean field games, has been developed extensively in \cite{carmona2018probabilistic, bensoussan2013mean, albi2017mean, fornasier2014mean, li2012stochastic}. It can be viewed as an optimal control problem \cite{yong1999stochastic} with dependence on distributions. These frameworks are particularly effective for modeling high-dimensional interacting systems. In recent years, there has been growing interest in connecting MFC and optimal control with machine learning algorithms, including flow matching methods \cite{ruthotto2020machine, guo2019learning, hu2023recent, reisinger2020rectified, hu2021mean}.

In approximating the FP equation for complex systems, the score function plays an important role. Various methods have been proposed to estimate the score function, including score matching \cite{hyvarinen2005estimation,10.5555/1795114.1795156}, kernel density estimation \cite{epstein2025scoredebiased,zhou2024deep}, and denoising autoencoder \cite{vincent2008extracting, vincent2011connection}. \cite{boffi2023probability} is closely related to our work. In that approach, the score function is directly parameterized and learned via a score-matching objective at each time step. In contrast, we parameterize the composed velocity field and evolve the score through a coupled ODE system, resulting in a trajectory-level velocity-matching formulation rather than a time-local score-matching procedure. A score-based framework for generative modeling, namely time-reversible diffusion models, is introduced in \cite{song2021scorebased}. This work is further extended to the Schr\"odinger bridges and stochastic control problems in \cite{de2021diffusion}. Normalizing flows \cite{rezende2015variational, papamakarios2021normalizing} and neural ordinary differential equations (neural ODEs) \cite{chen2018neural} enable continuous-time density estimation by learning deterministic transport maps. Recent work on flow matching replaces exact log-likelihood objectives with trajectory-level matching losses between learned and target dynamics \cite{lipman2023flow}. \cite{grathwohl2019ffjord} introduced scalable generative models based on neural ODEs with exact change-of-variable computation. Different from existing literature, we designed an algorithm that uses score-based normalizing flow to simulate the FP equation.

The rest of this paper is organized as follows. In Section \ref{sec:background}, we introduce the theoretical background, including the FP equation with entropy dissipation, formulation of flow matching for FP equations as an MFC problem, and the construction of score-based normalizing flow. In Section \ref{sec:algorithm}, we introduce the numerical algorithms for solving the flow matching problem. In Section \ref{sec:results}, we present numerical results on both low- and high-dimensional Fokker--Planck equations, including Langevin dynamics, underdamped Langevin dynamics, chaotic systems, and high-dimensional interacting particle systems, to demonstrate the effectiveness and scalability of the proposed method. In Section \ref{sec:convergence}, we provide a convergence analysis for the OU process. 

\section{Theoretical background}\label{sec:background}
In this section, we introduce the FP equation with its entropy dissipation property. Then we formulate flow matching for FP equations as MFC problems that involve the score functions. We introduce the score-based normalizing flow as an efficient way to compute the score function. We clarify some notations first. $\nx$ denotes the gradient or Jacobian matrix of a function w.r.t. the variable $x \in \RR^d$, where the gradient is always a column vector. $\nx\cdot$, $\nx^2$, and $\Delta_x$ are the divergence, the Hessian, and the Laplacian w.r.t. the variable $x$. $\abs{\cdot}$ is the absolute value of a scalar, the $l_2$ norm of a vector, or the Frobenius norm of a matrix. $\Tr(\cdot)$ denotes the trace of a square matrix. $\norm{\cdot}_2$ denotes the $l_2$ operator norm of a matrix. 

\subsection{Fokker--Planck equation and entropy dissipation}
Consider the FP equation
\begin{equation}\label{eq:FP_b}
\pt \rho(t,x) + \nx\cdot\parentheses{\rho(t,x) b(x)} = \ve \Delta_x \rho(t,x) ~,\quad \rho(0,\cdot) = \rho_0,
\end{equation}
where $b(x): \RR^d \to \RR^d$ is the drift function and $\ve > 0$ is the diffusion constant. $\rho(t,x)$ is the probability density function for the stochastic differential equation (SDE)
\begin{equation}\label{eq:SDE_x}
\rd x_t = b(x_t) \,\rd t + \sqrt{2\ve}\,\rd W_t ~, \quad x_0 \sim \rho_0.
\end{equation}
Throughout the paper, we assume $b(x)$ satisfies the local Lipschitz and Lyapunov condition
$$x\tp b(x) \le K(1 + |x|^2)$$
such that the SDE \eqref{eq:SDE_x} admits a unique global strong solution on any finite time interval $[0,T]$ \cite{mao2007stochastic}.
Since the FP equation \eqref{eq:FP_b} is time-homogeneous, i.e. $b(x)$ does not depend on $t$, the FP equation admits a stationary distribution $\pi(x)$ under mild regularity condition such as the Foster--Lyapunov criteria \cite{meyn1993stability}. This $\pi(x)$ satisfies the stationary FP equation
\begin{equation*}
\nx \cdot (\pi(x) b(x)) = \ve \Delta_x \pi(x).
\end{equation*}
We define the Kullback–-Leibler (KL) divergence as
\begin{equation*}
\KL{\rho}{\pi} := \int \rho(x) \log \dfrac{\rho(x)}{\pi(x)} \,\rd x.
\end{equation*}
Then, the FP equation satisfies the following entropy dissipation property.
\begin{proposition}[Dissipation of relative entropy]\label{prop:dissipation_entropy}
Let $\rho$ be the solution to the FP equation \eqref{eq:FP_b}, then
\begin{equation*}
\dfrac{\rd}{\rd t} \KL{\rho(t,\cdot)}{\pi} = - \ve \int \abs{\nx \log\dfrac{\rho(t,x)}{\pi(x)}}^2 \rho(t,x) \,\rd x.
\end{equation*}
\end{proposition}
Here, $\nx \log\dfrac{\rho(t,x)}{\pi(x)}$ is the relative score function.

A similar property holds for the ULDs. The ULD is an SDE that models the motion of a particle subject to deterministic forces, random thermal fluctuations, and damping effects \cite{cheng2018underdamped}. The dynamic is
\begin{equation}\label{eq:ULD}
\begin{cases}
\rd x_t = v_t \,\rd t\\
\rd v_t = -\parentheses{\gamma v_t + \nx U(x_t)} \,\rd t + \sqrt{2\gamma \beta^{-1}} \, \rd W_t.
\end{cases}
\end{equation}
Here, $\gamma$ is the damping coefficient, $\beta^{-1}$ is the temperature, and $U(x)$ is the potential function governing the deterministic forces. The FP equation for ULD with density function $\rho(t,x,v)$ is
\begin{equation}\label{eq:FP_ULD}
\pt\rho + \nx\cdot\parentheses{v\rho} - \nv\cdot\parentheses{(\gamma v+\nx U)\rho} = \frac{\gamma}{\beta} \Delta_v \rho.
\end{equation}
We define the Hamiltonian as $H(x,v) = \frac12 |v|^2 + U(x)$. The stationary distribution for ULD is
\begin{equation*}
\pi(x,v) = \dfrac{1}{Z} \exp\sqbra{-\beta\parentheses{\frac12|v|^2 + U(x)}} = \dfrac{1}{Z} \exp\parentheses{-\beta H(x,v)},
\end{equation*}
where $Z = \int_{\RR^{2d}} \exp\sqbra{-\beta\parentheses{\frac12|v|^2 + U(x)}} \rd x\,\rd v$ is the normalization constant. We define the free energy for a density function $\rho(x,v)$ as
\begin{equation}\label{eq:ULD_energy}
\Dh(\rho) = \int_{\RR^{2d}} \parentheses{\log\rho(x,v) + \beta H(x,v)} \rho(x,v) \,\rd x\,\rd v.
\end{equation}
We can also show the dissipation of energy along the dynamic.
\begin{proposition}[Free energy dissipation of ULD]\label{prop:ULD_dissipation}
The ULD satisfies the following energy dissipation formula
\begin{equation}\label{eq:ULD_dissipation}
\dfrac{\rd}{\rd t} \Dh(\rho(t,\cdot,\cdot)) = - \dfrac{\gamma}{\beta} \int_{\RR^{2d}} \rho(t,x,v) \abs{\nv\log\parentheses{\dfrac{\rho(t,x,v)}{\pi(x,v)}}}^2 \rd x \, \rd v.
\end{equation}
\end{proposition} 

\subsection{Mean field control problem for Fokker--Planck equations}\label{sec:flow_matching}
Flow matching for FP equation can be viewed as a MFC problem with a quadratic discrepancy cost
\begin{equation}\label{eq:MFC_v}
\inf_{v,\rho} \int_0^T \int_{\RR^d} \abs{v(t,x) - b(x)}^2 \rho(t,x) \,\rd x \,\rd t
\end{equation}
subject to the controlled FP equation
\begin{equation}\label{eq:FP_v}
\pt \rho(t,x) + \nx\cdot\parentheses{\rho(t,x) v(t,x)} = \ve \Delta_x \rho(t,x) ~,\quad \rho(0,\cdot) = \rho_0,
\end{equation}
where $T>0$ is the terminal time. The controlled state dynamic is characterized by the FP equation \eqref{eq:FP_v}, where one aims to match the target dynamic \eqref{eq:FP_b}. The optimal control field to this MFC problem is $v^*(t,x) = b(x)$.

In this work, we reformulate the flow matching problem using the score function. By the identity $\nx \rho / \rho = \nx \log \rho$, we can rewrite the FP equation \eqref{eq:FP_b} as
\begin{equation}\label{eq:blob}
\pt \rho(t,x) + \nx\cdot\sqbra{\parentheses{b(x) - \ve \nx\log\rho(t,x)}\rho(t,x)} = 0.
\end{equation}
Given the density function $\rho(t,x)$, we define the composed velocity field as $f(t,x) := b(x) - \ve \nx\log\rho(t,x)$, then the MFC problem for flow matching \eqref{eq:MFC_v} can be rewritten as
\begin{equation}\label{eq:MFC_f}
\inf_{f,\rho} \int_0^T \int_{\RR^d} \abs{f(t,x) - b(x) + \ve \nx\log\rho(t,x)}^2 \rho(t,x) \,\rd x \,\rd t
\end{equation}
subject to the controlled continuity equation
\begin{equation}\label{eq:continuity}
\pt \rho(t,x) + \nx\cdot\parentheses{\rho(t,x) f(t,x)} = 0 ~,\quad \rho(0,\cdot) = \rho_0.
\end{equation}
\begin{proposition}[Well-posedness of the MFC]
The mean field control problem \eqref{eq:MFC_f}--\eqref{eq:continuity} admits a unique solution pair $(f^*,\rho^*)$. In particular, $\rho^*$ coincides with the solution to the Fokker--Planck equation \eqref{eq:FP_b}, and the optimal velocity field is given by
\[
f^*(t,x) = b(x) - \varepsilon \nabla_x \log \rho^*(t,x).
\]
\end{proposition}
\begin{proof}
Let $\rho^*$ denote the solution to \eqref{eq:FP_b}, and define $f^*(t,x) := b(x) - \varepsilon \nabla_x \log \rho^*(t,x)$. By construction, the pair $(f^*,\rho^*)$ satisfies the continuity equation \eqref{eq:continuity}. Moreover, substituting $(f^*,\rho^*)$ into the objective functional \eqref{eq:MFC_f}, the integrand vanishes identically, and hence the minimal value of the cost is zero. Therefore, $(f^*,\rho^*)$ is an optimal solution.\\
We next prove uniqueness. Suppose $(f',\rho')$ is another optimal solution to \eqref{eq:MFC_f}--\eqref{eq:continuity}. The optimality implies the objective \eqref{eq:MFC_f} is $0$ and
\[
f'(t,x) = b(x) - \varepsilon \nabla_x \log \rho'(t,x).
\]
Substituting this into \eqref{eq:continuity}, we find that $\rho'$ satisfies the Fokker--Planck equation \eqref{eq:FP_b}.
\end{proof}

The advantage of this formulation is that the state dynamic is completely governed by a deterministic velocity field $f(t,x)$ without diffusion. Consequently, we are able to compute the reverse process, which has important applications in reversible diffusion \cite{agmon1990theory} and generative models \cite{song2021scorebased}. This formulation \eqref{eq:MFC_f} also incurs a computational challenge---approximating the score function $\nx \log\rho(t,x)$.

\subsection{The score-based normalizing flow}\label{sec:SBNF}
We introduce the score-based normalizing flow in this section. Consider a probability flow, where the state process is governed by the ODE $\pt x_t = f(t,x_t)$ with random initialization $x_0 \sim \rho_0$. Let $\rho(t,\cdot)$ be the probability density function for $x_t$, then $\rho(t,x)$ satisfies the continuity equation \eqref{eq:continuity}. We denote $L_t = \rho(t,x_t)$, $l_t = \log\rho(t,x_t)$, $s_t = \nx\log\rho(t,x_t)$, and $H_t = \nx^2 \log\rho(t,x_t)$, then they satisfy the following ODE systems \cite{shen2022self,boffi2023probability,zhou2025score}.
\begin{proposition}[Score-based neural ODE systems]\label{prop:ODE}
Let $x_t$ satisfy $\pt x_t = f(t,x_t)$. Then, $L_t = \rho(t,x_t)$, $l_t = \log\rho(t,x_t)$, $s_t = \nx\log\rho(t,x_t)$, and $H_t = \nx^2 \log\rho(t,x_t)$ satisfy
\begin{subequations}\label{eq:ODE_systems}
\begin{align}
\pt L_t &= -\nx \cdot f(t,x_t) L_t, \label{eq:Lt_ODE}\\
\pt l_t &= -\nx \cdot f(t,x_t), \label{eq:lt_ODE}\\
\pt s_t &= -\nx f(t,x_t)\tp s_t - \nx(\nx\cdot f(t,x_t)), \label{eq:st_ODE}\\
\pt H_t &= -\sum_{i=1}^d s_{it} \nx^2 f_i(t,x_t) - \nx^2(\nx\cdot f(t,x_t)) - H_t \nx f(t,x_t) -  \nx f(t,x_t) \tp H_t, \label{eq:Ht_ODE}
\end{align}
\end{subequations}
where $s_{it}$ and $f_i$ are the $i$-th component of $s_t$ and $f$ respectively.
\end{proposition}

With $f$ parametrized, we can compute the score function $s_t = \nx \log\rho(t,x_t)$ along the trajectory efficiently \cite{zhou2025score}.

\section{Numerical algorithm}\label{sec:algorithm}
In this section, we present the numerical algorithm for solving the MFC problem \eqref{eq:MFC_f}. We parametrize the composed velocity field $f(t,x)$ as a multilayer perceptron (MLP) neural network $f_\theta(t,x)$, where $\theta$ denotes its parameters. Throughout the paper, we assume the drift $b$ and the initial density $\rho_0$ are known.

\subsection{Numerical algorithm for flow matching}
Let $N_x$ and $N_t$ denote the number of samples and the number of sub-intervals for time. We partition the time interval $[0,T]$ into $N_t$ uniform subintervals with length $\dt = T / N_t$ and denote the discrete time points by $t_j = j \dt$ for $j=0,\ldots,N_t$. Given initial samples $\{x_0^{(n)}\}_{n=1}^{N_x}$ drawn i.i.d. from $\rho_0$, we simulate the ODE dynamics \eqref{eq:ODE_systems} through
\begin{subequations}
\begin{align}
x_{t_{j+1}} &= x_{t_j} + \dt \, f_\theta(t_j, x_{t_j}), \label{eq:xt_discrete}\\
L_{t_{j+1}} &= L_{t_j} - \dt \, \nx\cdot f_\theta(t_j, x_{t_j}) L_{t_j}, \label{eq:Lt_discrete}\\
l_{t_{j+1}} &= l_{t_j} - \dt \, \nx\cdot f_\theta(t_j, x_{t_j}), \label{eq:lt_discrete}\\
s_{t_{j+1}} &= s_{t_j} - \dt \parentheses{ \nx f_\theta(t_j, x_{t_j})\tp s_{t_j} + \nx( \nx\cdot f_\theta(t_j, x_{t_j})) }, \label{eq:st_discrete} \\
H_{t_{j+1}} &= H_{t_j} - \dt \left( \sum_{i=1}^d s_{it_j} \nx^2 (f_\theta)_i(t_j,x_{t_j}) + \nx^2(\nx\cdot f_\theta(t_j,x_{t_j})) \right. \label{eq:Ht_discrete}\\
& \hspace{0.9in}  + H_{t_j} \nx f_\theta(t_j,x_{t_j}) + \nx f_\theta(t_j,x_{t_j}) \tp H_{t_j}  \bigg),\nonumber
\end{align}
\end{subequations}
where the derivatives of $f$ are obtained through auto-differentiation. The variational objective \eqref{eq:MFC_f} is then approximated using Monte Carlo sampling and discretization
\begin{equation}\label{eq:loss_discrete}
L(\theta) = \dfrac{1}{N_x} \sum_{n=1}^{N_x} \sum_{j=0}^{N_t-1} \abs{ f_\theta(t_j,x_{t_j}^{(n)}) - b(x_{t_j}^{(n)}) + \ve s_{t_j}^{(n)}}^2 \dt,
\end{equation}
and minimized through backpropagation of the ODE system using Adam method \cite{Kingma2015adam}. We conclude this method in Algorithm \ref{alg:FM_one_interval}.

\begin{algorithm}[htb]
\caption{Flow matching solver through score-based normalizing flow}
\begin{algorithmic}[1]
\STATE \textbf{Input:}~ Flow matching problem, neural network structure, number of steps $N_{\text{step}}$, learning rate, $N_x,N_t$
\STATE \textbf{Output:}~ Approximated composed velocity field $f_\theta(\cdot,\cdot)$ for the MFC problem
\STATE \textbf{Initialization:}~Parameter $\theta$ for the composed neural network, $\dt=T/N_t$
\FOR{$\text{step}=1,2,\ldots, N_{\text{step}}$}
\STATE{Sample $\{x_0^{(n)}\}_{n=1}^{N_x}$ i.i.d. from $\rho_0$}
\STATE{Compute $s_0^{(n)} = \nx\log\rho_0(x_0^{(n)})$ for $n=1,\ldots,N_x$.}
\STATE{$\text{Loss} = 0$}
\FOR{$j=0,1,\ldots,N_t-1$}
\STATE{$\text{Loss} = \text{Loss} + \frac{1}{N_x} \sum\limits_{n=1}^{N_x} \abs{f_\theta(t_j,x^{(n)}_{t_j}) - b(x^{(n)}_{t_j}) + \ve s_{t_j}^{(n)}}^2\dt$}\hfill\COMMENT{update loss}
\STATE{$x^{(n)}_{t_{j+1}} = x^{(n)}_{t_j} + \dt \, f_\theta(t_j, x^{(n)}_{t_j})$}\hfill\COMMENT{update state}
\STATE{$s^{(n)}_{t_{j+1}} = s^{(n)}_{t_j} - \dt \parentheses{ \nx f_\theta(t_j, x^{(n)}_{t_j})\tp s^{(n)}_{t_j} + \nx( \nx\cdot f_\theta(t_j, x^{(n)}_{t_j})) }$}\hfill\COMMENT{update score}
\ENDFOR
\STATE{Update $\theta$ through Adam method to minimize $\text{Loss}$}
\ENDFOR
\end{algorithmic}
\label{alg:FM_one_interval}
\end{algorithm}

\subsection{Numerical algorithms for flow matching over long horizons}\label{sec:algorithm_long}
For FP equations with long time horizons, it is challenging to characterize the entire dynamic using a single neural network. To address this, we adopt a multi-stage approach by decomposing the time interval $[0,T]$ into $N_T$ stages (subintervals), each of length $T':=T / N_T$. We denote the stage endpoints by $T_m = m T'$ for $m=1,\ldots,N_T$, and define each stage interval as $I_m := [T_{m-1}, T_{m}]$. 

For each stage $I_m$,  we parameterize the velocity field by an independent neural network $f_m(t,x) := f_{\theta_m}(t,x)$. The flow matching model is trained sequentially for each stage. The initial states and score functions for stage $I_m$ are obtained by simulating the normalizing flow dynamics \eqref{eq:xt_discrete} and \eqref{eq:st_discrete} using the previously trained network. 

To improve efficiency, for stages $m=2,\ldots, N_T$, we initialize the parameter $\theta_m$ as $\theta_{m-1}$, from the preceding stage, following a warm-start strategy \cite{zhou2023neural}. This initialization provides a good starting point, often leading to faster convergence and reduced training iterations in subsequent stages. We summarize the algorithm for flow matching problems with long time horizons in Algorithm \ref{alg:FM_multi_interval}.

\begin{algorithm}[htb]
\caption{Flow matching solver for long time horizon}
\begin{algorithmic}[1]
\STATE \textbf{Input:}~ Flow matching problem, neural network structure for each stage, number of steps for the first stage $N_{\text{step}0}$ and the following stages $N_{\text{step}}$, learning rages, $N_x,N_t, N_T$
\STATE \textbf{Output:}~ Approximated composed velocity fields $\{f_m(\cdot,\cdot)\}_{m=1}^{N_T}$ for the MFC problem
\STATE \textbf{Initialization:}~Parameter $\theta_1$ for the first neural network $f_1(\cdot,\cdot)$
\STATE{Apply Algorithm \ref{alg:FM_one_interval} to train $f_1(\cdot,\cdot)$ for $N_{\text{step}0}$ steps with $\dt=T/(N_T N_t)$.}
\FOR{$m = 2,3,\ldots,N_T$} 
\STATE{Initialize $\theta_m$ as $\theta_{m-1}$ from previous stage}
\hfill \COMMENT{warm-start}
\FOR{$\text{step}=1,2,\ldots, N_{\text{step}}$}
\STATE{Sample $\{x_0^{(n)}\}_{n=1}^{N_x}$ i.i.d. from $\rho_0$}
\STATE{Compute $s_0^{(n)} = \nx\log\rho_0(x_0^{(n)})$ for $n=1,\ldots,N_x$.}
\FOR{$m' = 1,\ldots,m-1$}
\STATE{$t_0 = (m'-1)T'$}\hfill\COMMENT{initial time for the stage}
\FOR{$j=0,\ldots,N_t-1$}
\STATE{$t_{j+1} = ((m'-1)N_t+j+1)\dt$}\hfill\COMMENT{set time stamp}
\STATE{$x^{(n)}_{t_{j+1}} = x^{(n)}_{t_j} + \dt \, f_{m'}(t_j, x^{(n)}_{t_j})$}\hfill\COMMENT{update state}
\STATE{$s^{(n)}_{t_{j+1}} = s^{(n)}_{t_j} - \dt \parentheses{ \nx f_{m'}(t_j, x^{(n)}_{t_j})\tp s^{(n)}_{t_j} + \nx( \nx\cdot f_{m'}(t_j, x^{(n)}_{t_j})) }$}\hfill\COMMENT{update score}
\ENDFOR
\ENDFOR
\STATE{$\text{Loss} = 0$}\hfill\COMMENT{compute loss for stage $m$}
\STATE{$t_0 = (m-1)T'$}
\FOR{$j=0,1,\ldots,N_t-1$}
\STATE{$\text{Loss} = \text{Loss} + \frac{1}{N_x} \sum_{n=1}^{N_x} \abs{f_m(t_j,x^{(n)}_{t_j}) - b(x^{(n)}_{t_j}) + \ve s_{t_j}^{(n)}}^2\dt$}\hfill\COMMENT{update loss}
\STATE{$t_{j+1} = ((m-1)N_t+j+1)\dt$}\hfill\COMMENT{set time stamp}
\STATE{$x^{(n)}_{t_{j+1}} = x^{(n)}_{t_j} + \dt \, f_m(t_j, x^{(n)}_{t_j})$}\hfill\COMMENT{update state}
\STATE{$s^{(n)}_{t_{j+1}} = s^{(n)}_{t_j} - \dt \parentheses{ \nx f_m(t_j, x^{(n)}_{t_j})\tp s^{(n)}_{t_j} + \nx( \nx\cdot f_m(t_j, x^{(n)}_{t_j})) }$}\hfill\COMMENT{update score}
\ENDFOR
\STATE{Update $\theta_m$ through Adam method to minimize $\text{Loss}$}
\ENDFOR
\ENDFOR
\end{algorithmic}
\label{alg:FM_multi_interval}
\end{algorithm}

\section{Numerical results}\label{sec:results}
In this section, we present numerical results including Langevin dynamics, ULDs, chaotic systems, and interacting particle systems. All experiments were conducted on an NVIDIA TITAN V GPU using NVIDIA driver version 535.183.01 and CUDA 12.2. For all the problems, we apply a step size of $\dt=0.01$.

\subsection{Langevin dynamic}\label{sec:Langevin}
We consider the Langevin dynamic \eqref{eq:SDE_x} with $d=2n$ and $b(x) = -(I_d + c J_d) \nx V(x)$, where
\begin{equation*}
J_d = \begin{bmatrix} 0& I_n\\-I_n&0 \end{bmatrix}
\end{equation*}
is skew symmetric, $c=0.5$, and $V(x)$ is the potential function such that $Z := \int_{\RR^d} \exp(-V(x) / \ve) \,\rd x$ is finite. The invariant distribution is $\pi(x) = \exp(-V(x)/\ve)/Z$. We test Algorithm \ref{alg:FM_one_interval} with two potential functions: a quadratic function $V(x) = \frac12 |x|^2$ and a doublewell potential function $V(x) = \frac14 |x-c_1|^2 |x-c_2|^2$, where $c_1$ and $c_2$ are vectors in $\RR^d$ with all entries being $1$ and $-1$ respectively.

We define the Gibbs free energy functional as
\begin{equation}\label{eq:free_energy}
\DV(\rho(t,\cdot)) = \int_{\RR^d} \parentheses{ \log\rho(t,x) + \frac{1}{\ve} V(x)} \rho(t,x) \,\rd x.
\end{equation}
Then we can verify that
\begin{equation}\label{eq:normalization_constant}
\KL{\rho(t,\cdot)}{\pi} = \DV(\rho(t,\cdot)) + \log Z,
\end{equation}
where $Z$ is the partition function. By Proposition \ref{prop:dissipation_entropy},
\begin{equation}\label{eq:energy_dissipation}
\dfrac{\rd}{\rd t} \DV(\rho(t,\cdot)) = - \ve \int \abs{\nx \log\dfrac{\rho(t,x)}{\pi(x)}}^2 \rho(t,x) \,\rd x.
\end{equation}

\subsubsection{Langevin dynamic with quadratic potential}\label{sec:LangevinGaussian}
The numerical results for Langevin dynamic in $2$ dimensions with a quadratic potential function $V(x) = \frac12 |x|^2$ and $\ve=0.5$ are presented in Figure \ref{fig:LangevinOU}. The first figure in the first row shows the training curve in $\log 10$ scale. The discrepancy loss reaches an order of $\num{e-4}$. The shaded region in the loss curve represents the standard deviation observed during $10$ independent runs. The second and third figures on the first row show the free energy and its dissipation along time.
The free energy \eqref{eq:free_energy} and its dissipation \eqref{eq:energy_dissipation} are approximated numerically through
\begin{equation}\label{eq:free_energy_discrete}
\widehat{\DV}(\rho(t_j,\cdot)) = \dfrac{1}{N_x} \sum_{n=1}^{N_x} \parentheses{l^{(n)}_{t_j} + \frac{1}{\ve} V(x^{(n)}_{t_j})},
\end{equation}
and
\begin{equation}\label{eq:energy_dissipation_discrete}
\widehat{\pt \DV}(\rho(t_j,\cdot)) = - \dfrac{\ve}{N_x} \sum_{n=1}^{N_x} \abs{s^{(n)}_{t_j} - \nx\log \pi(x^{(n)}_{t_j})}^2.
\end{equation}
The approximation of the free energy over time, plotted in blue line in the second figure, is decaying and is aligned with the true energy shown as a dashed orange line. The approximated dissipation, plotted in blue in the third figure, is getting closer to $0$, and is aligned with the true value. The visualization of $f_\theta$ at $t=0$ is shown in the second row in Figure \ref{fig:LangevinOU}. The trained neural network accurately captures the true velocity field. We also report the errors for the neural network, the density and the score function: 
\begin{equation*}
\begin{aligned}
\text{err}_{f} &= \dfrac{1}{N_x}  \dfrac{1}{N_t+1} \sum_{n=1}^{N_x} \sum_{j=0}^{N_t} \abs{f\parentheses{t_j, x^{(n)}_{t_j};\theta} - f\parentheses{t_j, x^{(n)}_{t_j}}},\\
\text{err}_{\rho} &= \dfrac{1}{N_x} \dfrac{1}{N_t} \sum_{n=1}^{N_x}\sum_{j=1}^{N_t}  \abs{L^{(n)}_{t_{j}} - \rho(t_{j},x^{(n)}_{t_{j}})}, \\
\text{err}_{s} &= \dfrac{1}{N_x} \dfrac{1}{N_t} \sum_{n=1}^{N_x}\sum_{j=1}^{N_t}  \abs{ s^{(n)}_{t_{j}} - \nx \log \rho\parentheses{t_{j}, x^{(n)}_{t_{j}}}}.
\end{aligned}
\end{equation*}
Note that the errors for the density and score at $t=0$ is $0$, so we start at $j=1$. After training, the errors achieve $\text{err}_{f} = \num{1.68e-2}$, $\text{err}_{\rho} = \num{6.93e-3}$, and $\text{err}_{s} = \num{1.06e-2}$. The training time is $529$ seconds. In addition, equation \eqref{eq:normalization_constant} provides a way to estimate the normalization constant $Z$. When $T$ is sufficient large, $\KL{\rho(T,\cdot)}{\pi} \approx 0$ and hence
\begin{equation}\label{eq:estimation_Z}
Z \approx \exp\parentheses{- \DV(\rho(T,\cdot))}.
\end{equation}
Using this formulation, we obtain an estimation $\exp(-\widehat{\DV}(\rho(T,\cdot))) = 3.1091$ (against the true value $\pi$), with a relative error of $1.03\%$.

\begin{figure}[!htbp]
\centering
\includegraphics[width=0.35\linewidth]{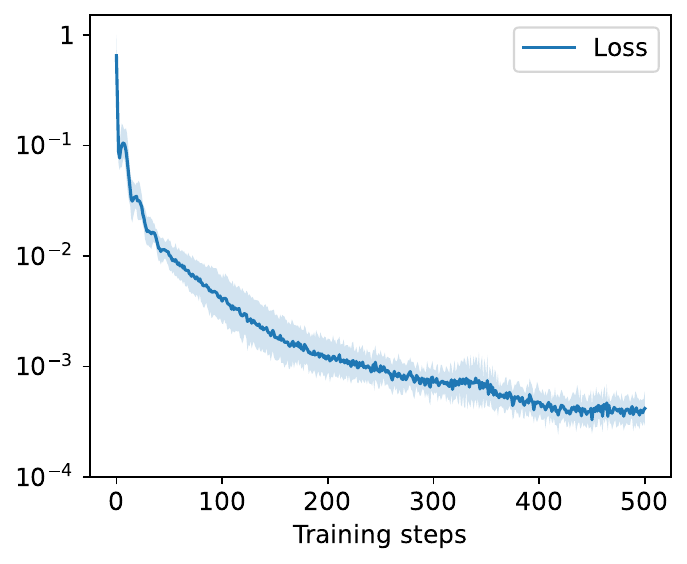}
\includegraphics[width=0.64\linewidth]{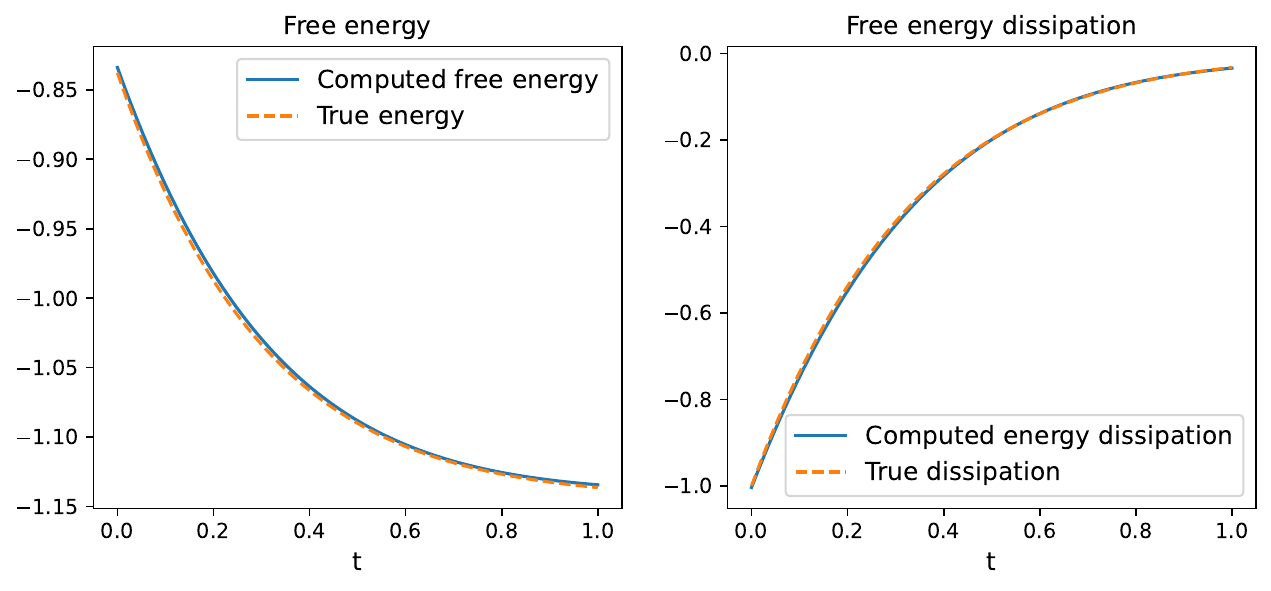}
\includegraphics[width=0.9\linewidth]{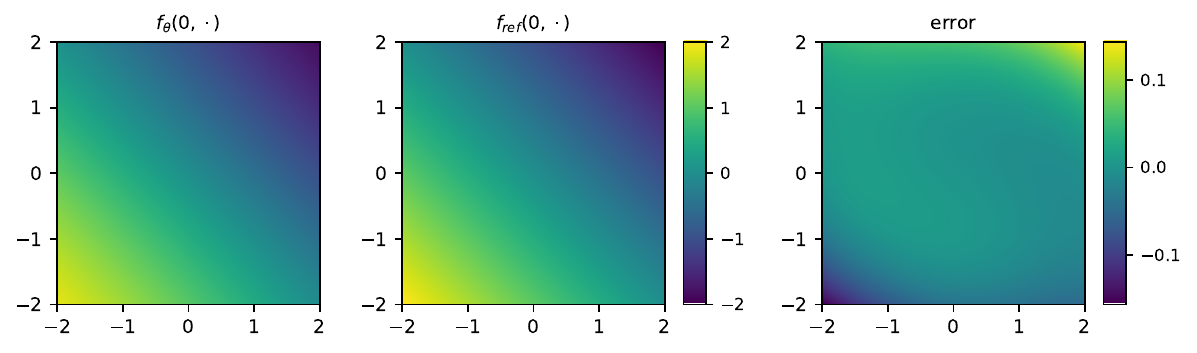}
\caption{Numerical results for Langevin dynamic with quadratic potential. First row: the curve for loss function in $\log 10$ scale, the curve for free energy and its dissipation along time. Second row: display the neural network, the reference solution, and the error at $t=0$ as images.}
\label{fig:LangevinOU}
\end{figure}

\subsubsection{Langevin dynamic with double-well potential}
Next, we consider the Langevin dynamic with double-well potential function
$V(x) = \frac14 |x-c_1|^2 |x-c_2|^2$, where $c_1$ and $c_2$ are vectors in $\RR^d$ with all entries being $1$ and $-1$ respectively.

The training time is $422$ seconds. The numerical results in $2$ dimensions are shown in Figure \ref{fig:LangevinDW}. The first figure shows the scattered plot of the particles under the trained ODE dynamic \eqref{eq:xt_discrete} at terminal time $T=1$ and the level sets of the stationary density function $\pi(x)$. The second figure shows the 3D histogram plot for the particles at $T$ and the surface plot for $\pi(x)$. We observe that $x_t$ correctly demonstrates the bifurcation phenomenon and splits into two piles, which coincide with the invariant distribution. The third and fourth figures show the free energy and its dissipation, computed through \eqref{eq:free_energy_discrete} and \eqref{eq:energy_dissipation_discrete}. In addition, we also report that the estimated normalization constant through \eqref{eq:estimation_Z} is $1.81726$, while a reference value obtained from Riemann sum is $1.83388$. The relative error for this estimation is $0.906\%$.

\begin{figure}[!htbp]
\centering
\includegraphics[width=0.49\linewidth]{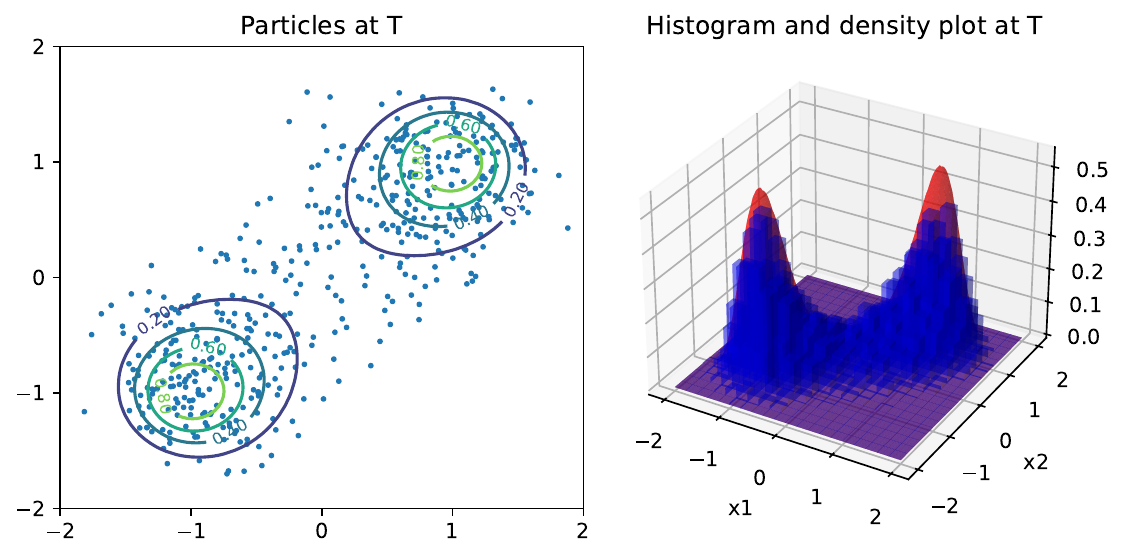}
\includegraphics[width=0.49\linewidth]{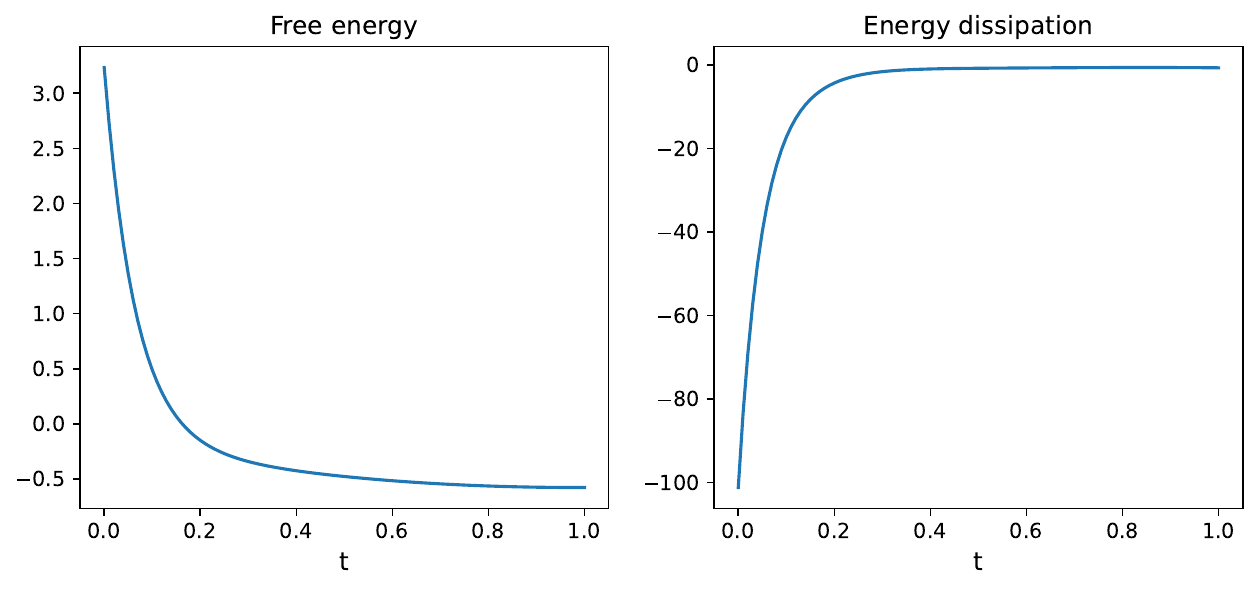}
\caption{Numerical results for Langevin dynamic with double-well potential. The first two figures show the scattered plot and histogram of the particles at terminal time $T$. The third and fourth figures show the free energy and its dissipation.}
\label{fig:LangevinDW}
\end{figure}

\subsection{Underdamped Langevin dynamic}
We present the numerical results for the ULDs in this section. Analogous to \eqref{eq:blob}, the FP equation for ULD \eqref{eq:FP_ULD} can be rewritten as
$$\pt\rho + \nx\cdot\parentheses{v\rho} - \nv\cdot\parentheses{(\gamma v+\nx U + \frac{\gamma}{\beta} \nv\log\rho)\rho} = 0.$$ 
Since there is no noise for $x_t$, we parametrize the composed velocity for $v$ as $\fvt(t,x,v)$, and formulate the flow matching MFC problem for ULD as
\begin{align*}
\inf_{f_v,\rho} \,& \int_0^T \int_{\RR^{2d}} \abs{f_v(t,x,v) + (\gamma v + \nx U(x)) + \gamma \beta^{-1} \nv\log\rho(t,x,v)}^2 \rho(t,x,v) \,\rd x \,\rd v \,\rd t\\
\text{s.t.} \,\,& \pt \rho(t,x) + \nx\cdot\parentheses{\rho(t,x,v) v} + \nv\cdot\parentheses{\rho(t,x,v) f_v(t,x,v)} = 0 ~,\quad \rho(0,\cdot,\cdot) = \rho_0.
\end{align*}
Since $\KL{\rho(t,\cdot,\cdot)}{\pi} = \Dh(\rho(t,\cdot,\cdot)) + \log Z$, we approximate the normalization constant through
\begin{equation}\label{eq:ULD_Z_approximation}
Z \approx \exp\parentheses{- \Dh(\rho(T,\cdot,\cdot))},
\end{equation}
where $T$ is sufficiently large.

We present the numerical results for a quadratic potential $U(x) = \frac12|x|^2$ and double-well potential $U(x) = \frac14 |x-c_1|^2 |x-c_2|^2$, which are the same as the potential functions in Section \ref{sec:Langevin}, with dimensions $2d=2$. The numerical implementation is in the same spirit as Algorithm \ref{alg:FM_multi_interval} in Section \ref{sec:algorithm}, and the details are deferred to the \apdx.

\subsubsection{Underdamped Langevin dynamic with quadratic potential}
The numerical results for $U(x)=\frac12 |x|^2$ are presented in Figure \ref{fig:ULD_Gaussian}. The training time is 779 seconds. The first row shows the training curve, the free energy and its dissipation. The discrepancy loss is less than $\num{1e-3}$. The approximation for the free energy and its dissipation accurately captures the reference value. The second and third rows in Figure \ref{fig:ULD_Gaussian} show the trained neural network at time $0$ and $T=5$, which accurately captures the true value. We also report that the errors for the neural network, density and score functions are $\text{err}_{f} = \num{1.58e-2}$, $\text{err}_{\rho} = \num{1.83e-3}$, $\text{err}_{s} = \num{3.73e-2}$. Additionally, the estimation for the normalization constant through \eqref{eq:ULD_Z_approximation} is $6.27844$ against the true value $2\pi$, with a relative error of $0.475\%$.

\begin{figure}[!htbp]
\centering
\includegraphics[width=0.34\linewidth]{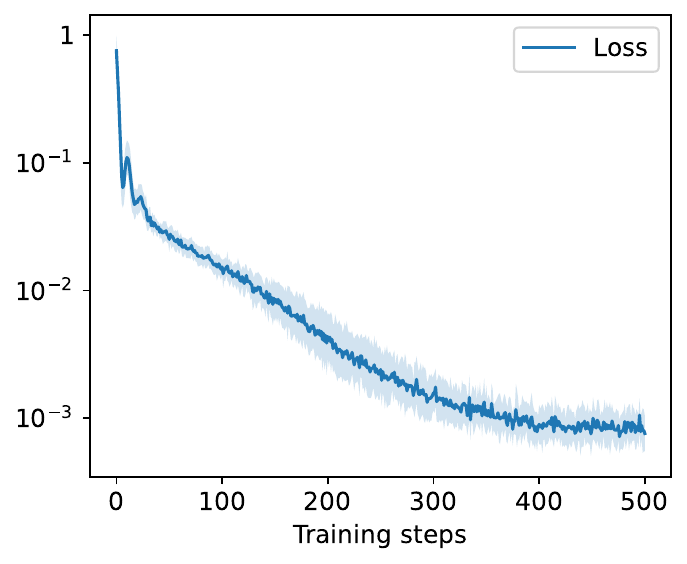}
\includegraphics[width=0.64\linewidth]{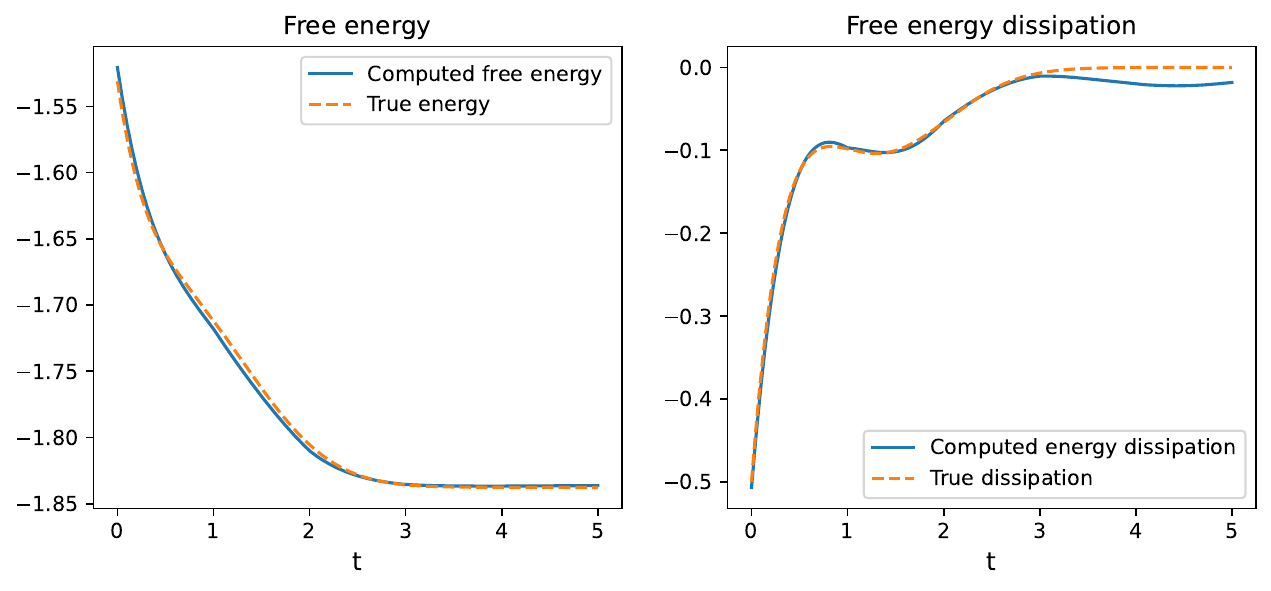}
\includegraphics[width=0.9\linewidth]{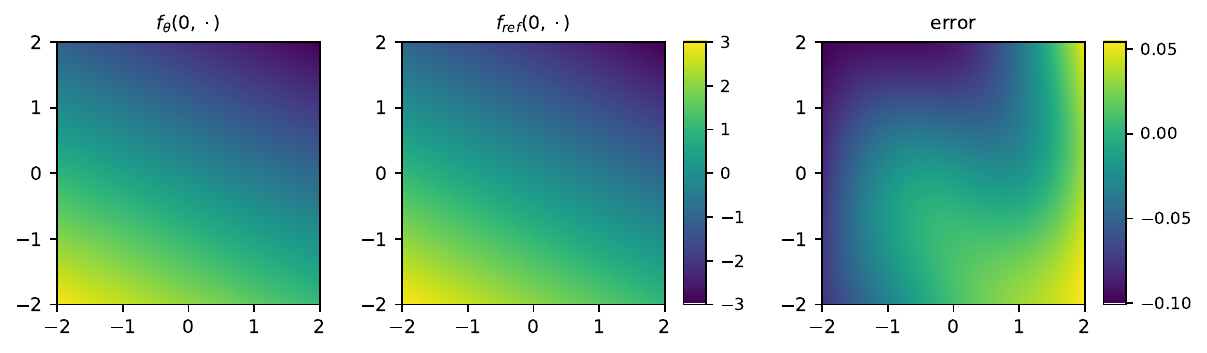}
\includegraphics[width=0.9\linewidth]{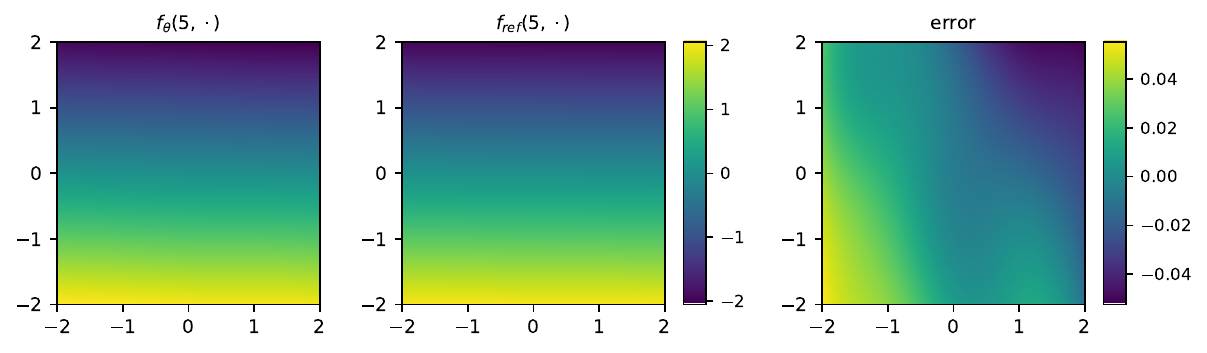}
\caption{Underdamped Langevin dynamic with quadratic potential. First row: learning curve, computed evolution of the free energy, and its dissipation along time. Second and third rows: display the neural network, the reference solution, and the error as images at $t=0$ and $t=5$.}
\label{fig:ULD_Gaussian}
\end{figure}

\subsubsection{Underdamped Langevin dynamic with double-well potential}
Next, we present the numerical result for underdamped Langevin dynamics with a double-well potential function. Since there is no reference solution, we compare the behavior of the dynamic \eqref{eq:xt_discrete} with the discretized Langevin dynamic. The training time is $667$ seconds. The numerical results is presented in Figure \ref{fig:ULDDW}, where each figure shows the density and velocity field of the particles. The first row shows the deterministic probability flow dynamic under the trained velocity field. We observe that the deterministic dynamic correctly captures the density evolution of ULD dynamic. Additionally, the velocity field for the deterministic dynamic is more organized than the stochastic ULD due to the absence of Brownian motion.

\begin{figure}[!htbp]
    \centering
    \includegraphics[width=\linewidth]{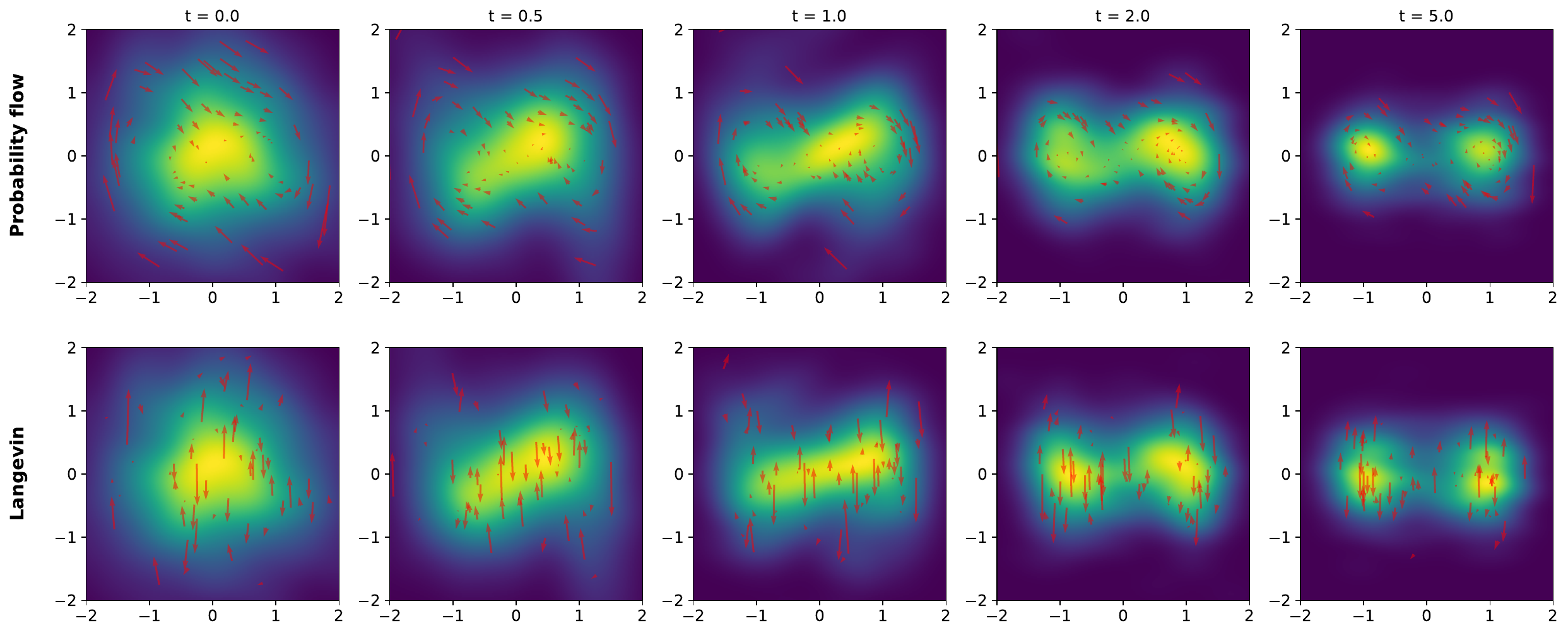}
    \caption{Density plot and velocity field for the Underdamped Langevin dynamic with double-well potential. First row: probability flow for deterministic dynamic with trained velocity field. Second row: stochastic dynamic \eqref{eq:ULD}.}
    \label{fig:ULDDW}
\end{figure}

\subsection{Chaotic systems}\label{sec:chaotic}
We present numerical results for three chaotic systems in this section, including the Lorenz system, the Arctangent Lorenz system, and the stochastic van der Pol oscillator. These chaotic systems have been intensively studied \cite{brunton2016discovering}. In all three systems, we add a diffusion with $\ve=0.1$.

\subsubsection{Lorenz system}
The Lorenz 63 system
\begin{equation}\label{eq:Lorenz63}
\left\{\begin{aligned}
\pt x_t=& \sigma(y_t-x_t),\\
\pt y_t=& x_t(\rho-z_t)-y_t,\\
\pt z_t=& x_ty_t-\beta z_t,
\end{aligned}\right.
\end{equation}
was first introduced by Edward Lorenz for modeling atmospheric convection \cite{lorenz2017deterministic}, with commonly used parameters $\sigma=10$, $\rho=28$, and $\beta=\frac83$. We implement a scaling (with parameter $s=0.2$) to the system \eqref{eq:Lorenz63} while preserving its chaotic behavior. In real world modeling, measurement noise or unpredictable external forcing is usually unavoidable. To reflect this, we add independent Brownian noise to each coordinate of the system with noise amplitude $\ve=0.1$. The resulting scaled stochastic Lorenz system in $\RR^3$ is
\begin{equation}\label{eq:Lorenz}
\left\{\begin{aligned}
\rd x_t=& \sigma(y_t-x_t) \,\rd t + \sqrt{2\ve} \,\rd W^x_t ,\\
\rd y_t=& \parentheses{x_t(\rho-z_t/s)-y_t} \,\rd t + \sqrt{2\ve} \,\rd W^y_t,\\
\rd z_t=& \parentheses{x_t y_t/s - \beta z_t} \,\rd t + \sqrt{2\ve} \,\rd W^z_t.
\end{aligned}\right.
\end{equation}
Here, $W^x_t$, $W^y_t$, and $W^z_t$ are independent standard Brownian motions for $x$, $y$, and $z$ coordinates respectively. We remark that when $\ve=0$, i.e., in the absence of noise, the scaled Lorenz system behaves identically to the original unscaled dynamics \eqref{eq:Lorenz63}. The scaling confines the trajectories within a more compact domain, which helps stabilize training by preventing divergence in the loss.

The training time is $2765$ seconds. The numerical results are shown in Figure
\ref{fig:Lorenz}. Density plots and the corresponding velocity fields are shown at selected time points $t=0,0.3,0.6,1.0,2.0,5.0$, chosen to highlight key stages of the dynamical evolution. At the terminal time $T=5$, the system appears to approach a stationary regime. The projections of the learned deterministic probability flow and the reference Langevin dynamics onto the $x$-$y$, $x$-$z$, and $y$-$z$ planes are shown in rows $1$-$2$, rows $3$-$4$, and rows $5$-$6$ of Figure \ref{fig:Lorenz}, respectively. We observe that the deterministic probability flow correctly characterizes the density evolution of the stochastic Lorenz dynamic. Additionally, the velocity field for the deterministic dynamic is more organized compared with the stochastic dynamic, reflecting the absence of diffusion.

\begin{figure}[!htbp]
\centering
\includegraphics[width=\linewidth]{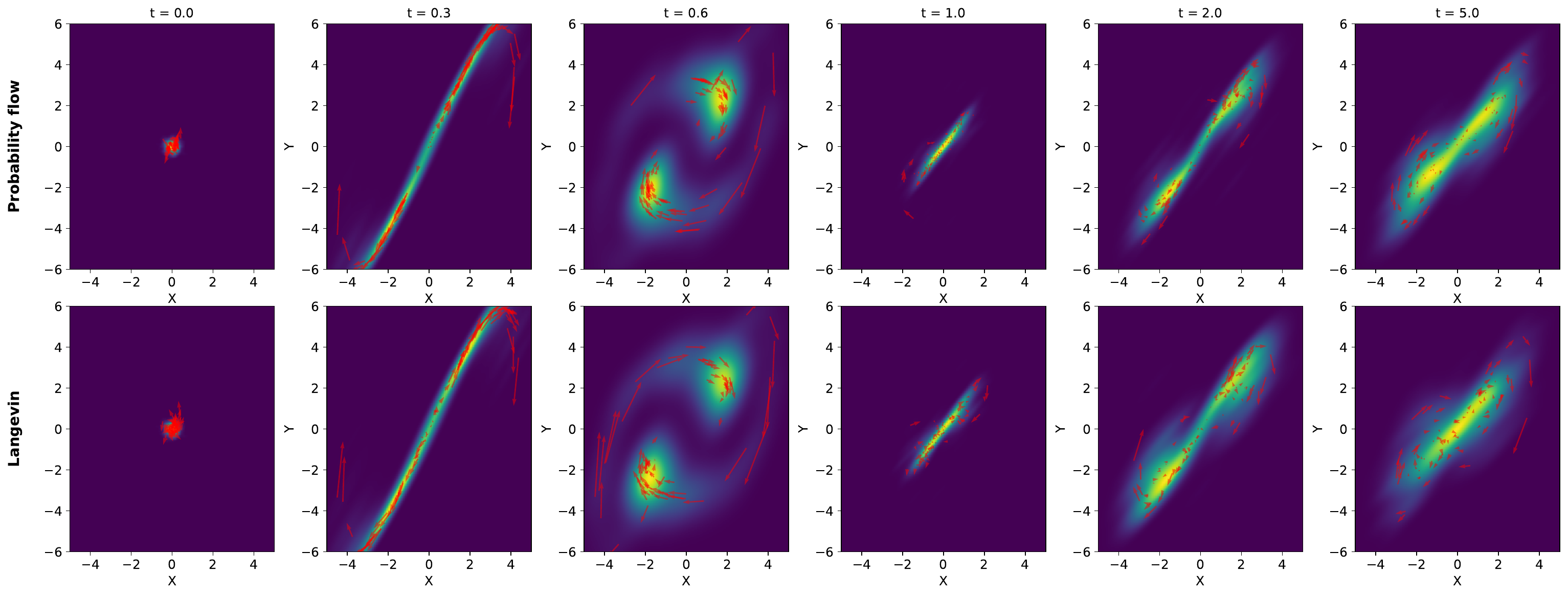}
\includegraphics[width=\linewidth]{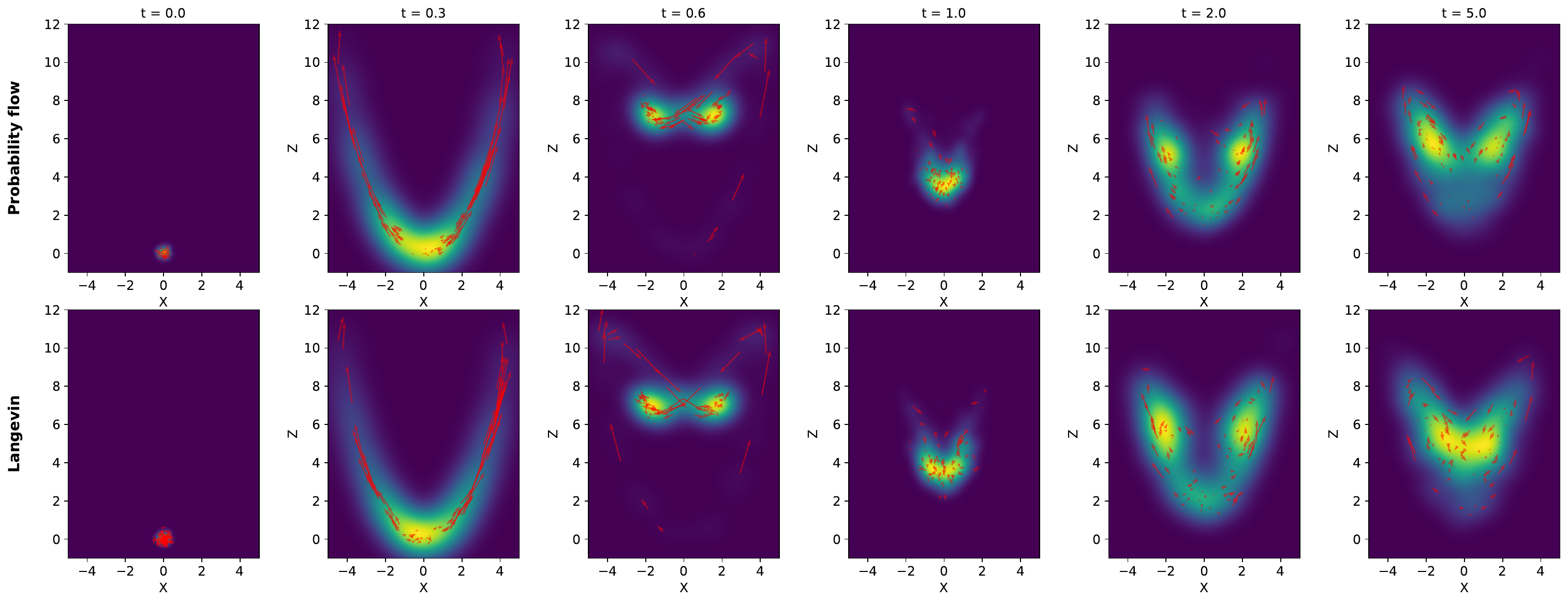}
\includegraphics[width=\linewidth]{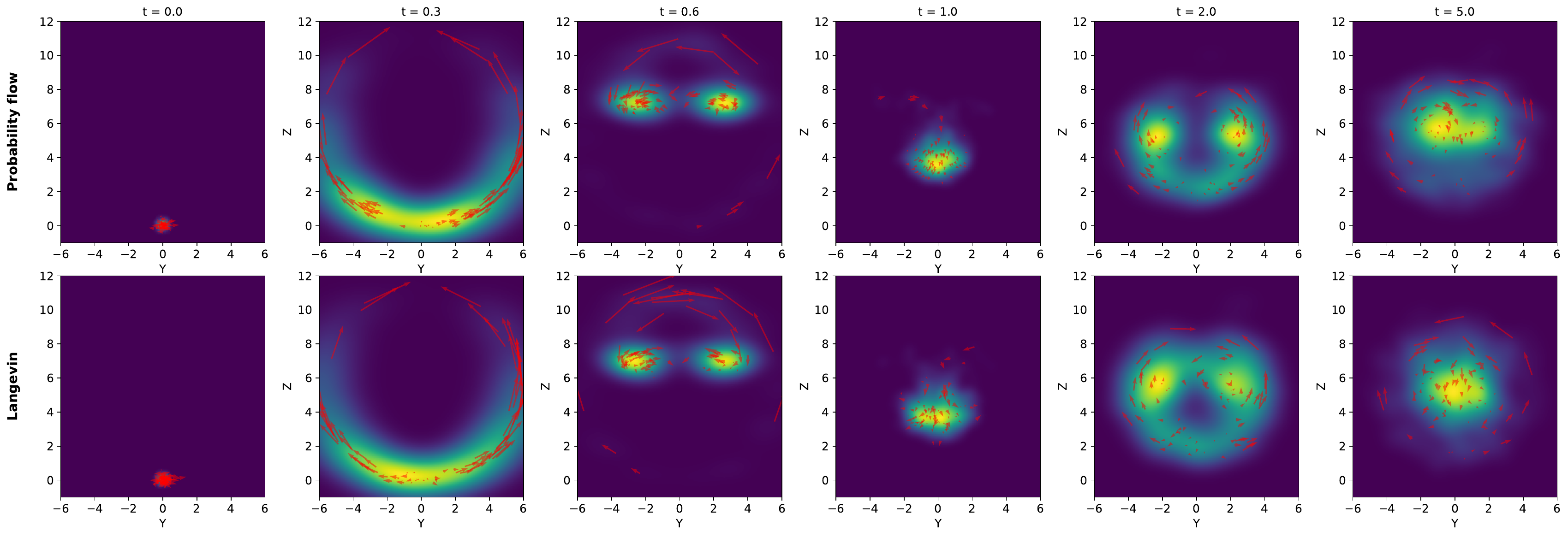}
\caption{Density plot and velocity field for the stochastic Lorenz system. Rows 1, 3, 5: projection of the deterministic dynamic under trained velocity field into $x$-$y$, $x$-$z$, and $y$-$z$ planes. Rows 2, 4, 6: projection of the stochastic dynamic \eqref{eq:Lorenz} into $x$-$y$, $x$-$z$, and $y$-$z$ planes.}
\label{fig:Lorenz}
\end{figure}

\subsubsection{Arctangent Lorenz system}
The Arctangent Lorenz system is studied in \cite{yang2023optimal}, where they have a scaling parameter $50$. We scale their system by $s=0.1$ and add a diffusion process as noise. Similar to the Lorenz example, we obtain the scaled stochastic Arctangent Lorenz system in $\RR^3$.
\begin{equation}\label{eq:AtanLorenz}
\left\{\begin{aligned}
\rd x_t=& 50s \arctan\parentheses{\sigma(y-x)/(50s)} \,\rd t + \sqrt{2\ve} \,\rd W^x_t ,\\
\rd y_t=& 50s \arctan\parentheses{x_t(\rho-50s z_t)- 50s y_t} \,\rd t + \sqrt{2\ve} \,\rd W^y_t,\\
\rd z_t=& 50s \arctan\parentheses{(x_t y_t/s - \beta z_t)/(50s)} \,\rd t + \sqrt{2\ve} \,\rd W^z_t.
\end{aligned}\right.
\end{equation}
The parameters are still $\sigma=10$, $\rho=28$, and $\beta=\frac83$, and the noise level is $\ve=0.1$.

The training time is $3205$ seconds. The numerical results at time stamps $t=0,0.5,0.7,1.0,2.0,5.0$ are presented in Figure \ref{fig:AtanLorenz}, whose layout is the same as Figure \ref{fig:Lorenz}. Similar to the Lorenz system, our trained velocity field $f_\theta$ is able to capture the stochastic process \eqref{eq:AtanLorenz} using the deterministic dynamic, which is validated by the density plot. In addition, the deterministic dynamic demonstrates a more structured velocity field, compared with the stochastic Langevin dynamic.

\begin{figure}[!htbp]
\centering
\includegraphics[width=\linewidth]{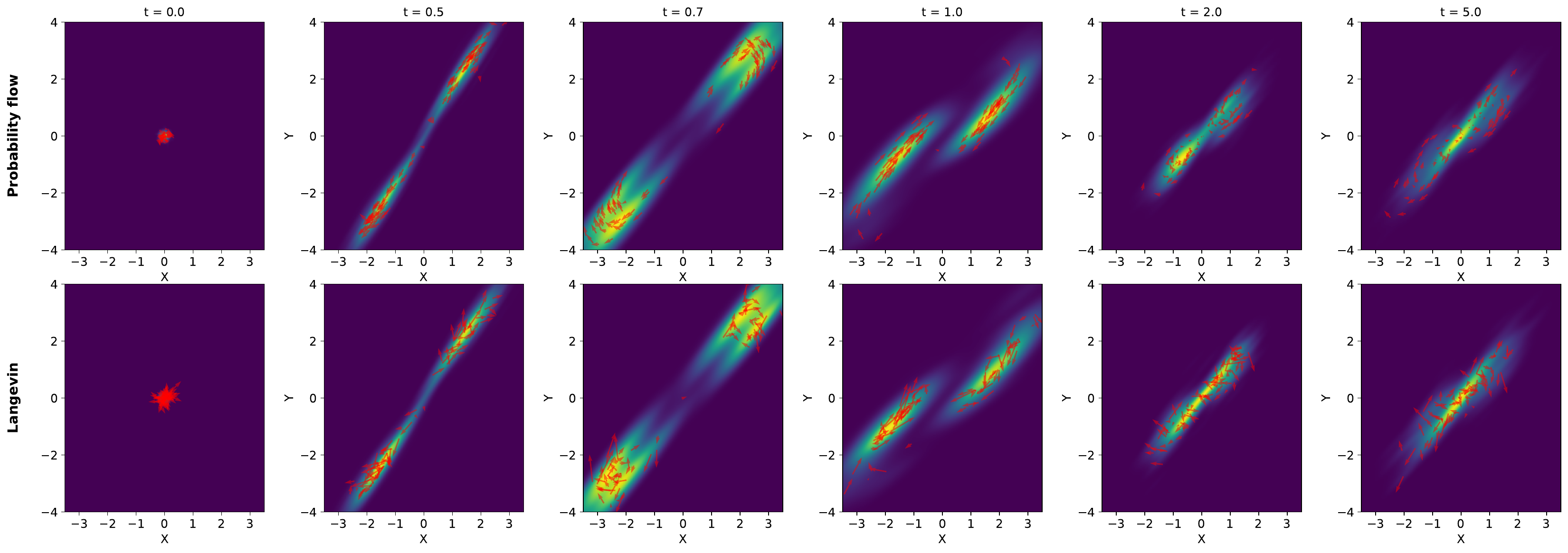}
\includegraphics[width=\linewidth]{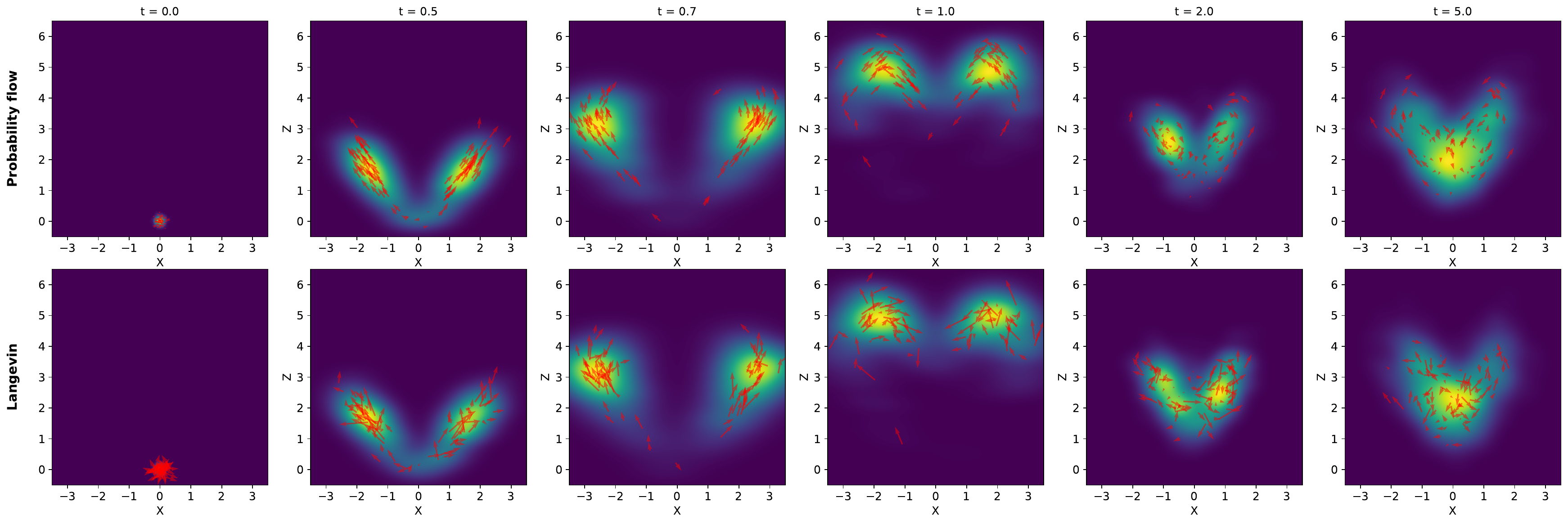}
\includegraphics[width=\linewidth]{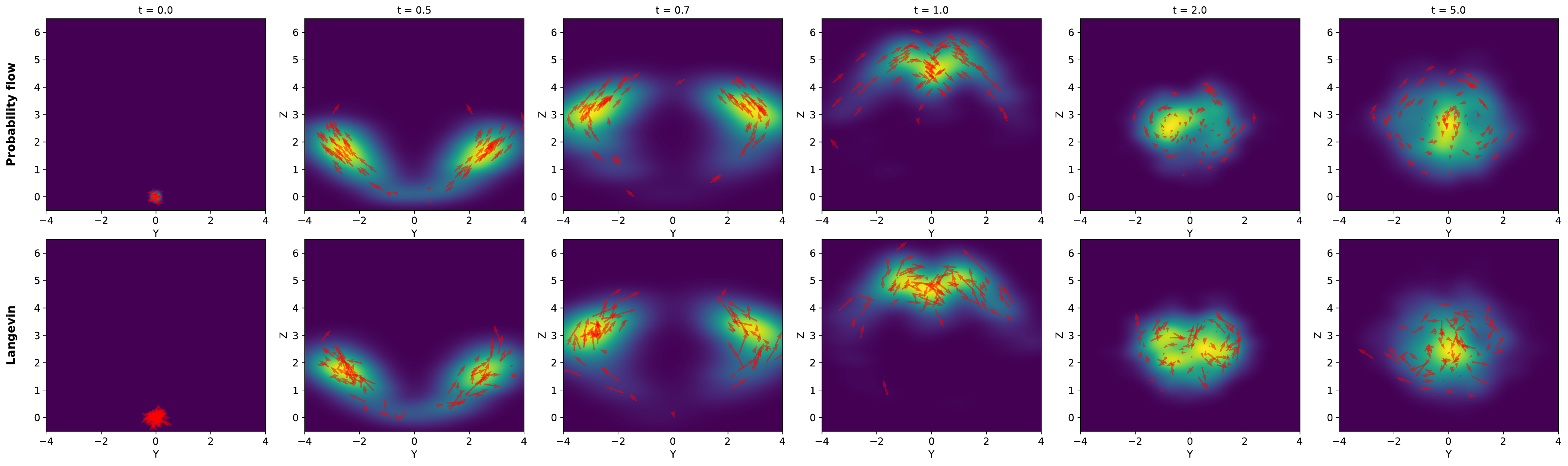}
\caption{Density plot and velocity field for the stochastic Arctangent Lorenz system. Rows 1, 3, 5: projection of the deterministic dynamic under trained velocity field into $x$-$y$, $x$-$z$, and $y$-$z$ planes. Rows 2, 4, 6: projection of the stochastic dynamic \eqref{eq:AtanLorenz} into $x$-$y$, $x$-$z$, and $y$-$z$ planes.}
\label{fig:AtanLorenz}
\end{figure}

\subsubsection{Stochastic van der Pol oscillator}\label{sec:VDP}
The van der Pol oscillator was first invented in the 1920s to model electrical circuits containing vacuum tubes. Then it has become one of the most studied nonlinear oscillator. A stochastic version of van der Pol oscillator is studied in \cite{xu2011stochastic,zhou2021actor}. In this work, we consider the stochastic van der Pol oscillator
\begin{equation}\label{eq:VDP_stochastic}
\begin{cases}
\rd x_t = v_t \,\rd t + \sqrt{2\ve} \,\rd W_t^x\\
\rd v_t = \mu (1-x_t^2) v_t \,\rd t + \sqrt{2\ve} \,\rd W_t^v.
\end{cases}
\end{equation}
with a scalar parameter $\mu=2.0$, and a diffusion parameter $\ve=0.1$ in $\RR^2$. The training time is $1405$ seconds. The numerical results are presented in Figure \ref{fig:VDP}, where we pick the time stamps $t=0.0, 0.2, 0.5, 1.0, 2.0$. The deterministic dynamic under the trained neural network correctly captures the density evolution while maintaining an organized velocity field.

\begin{figure}[!htbp]
\centering
\includegraphics[width=\linewidth]{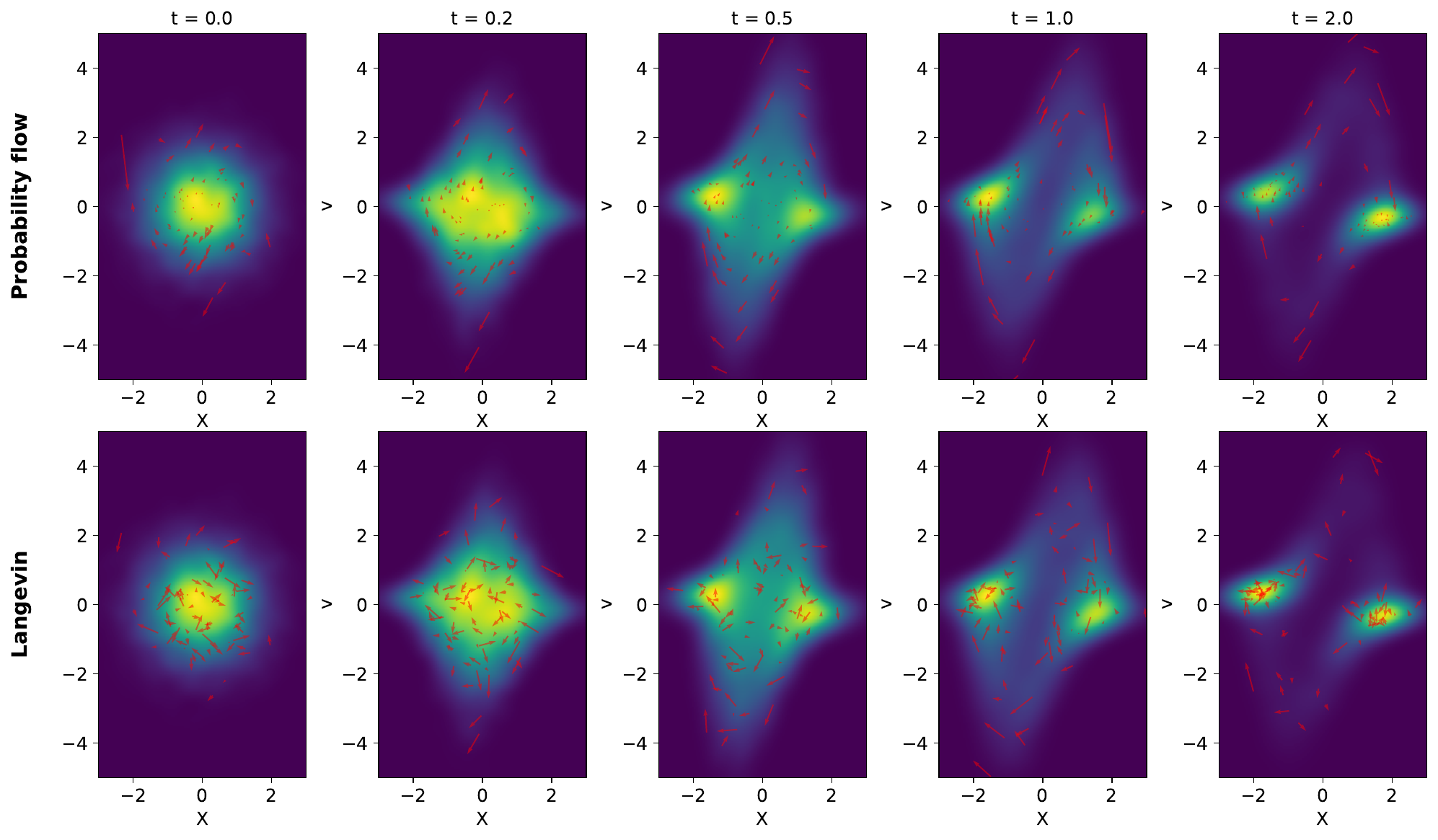}
\caption{Numerical result for stochastic van der Pol dynamics. Row 1: deterministic dynamic under trained velocity field. Row 2: reference stochastic dynamic \eqref{eq:VDP_stochastic}.}
\label{fig:VDP}
\end{figure}

\subsubsection{An active swimmer}
In this section, we consider the active swimmer model studied in \cite{boffi2023probability}. The stochastic dynamics are given by
\begin{equation}\label{eq:Swimmer_stochastic}
\begin{cases}
\rd x_t = (v_t- x_t^3) \,\rd t\\
\rd v_t = -\gamma v_t \,\rd t + \sqrt{2\gamma\ve} \,\rd W_t.
\end{cases}
\end{equation}

\begin{figure}[!ht]
\centering
\includegraphics[width=\linewidth]{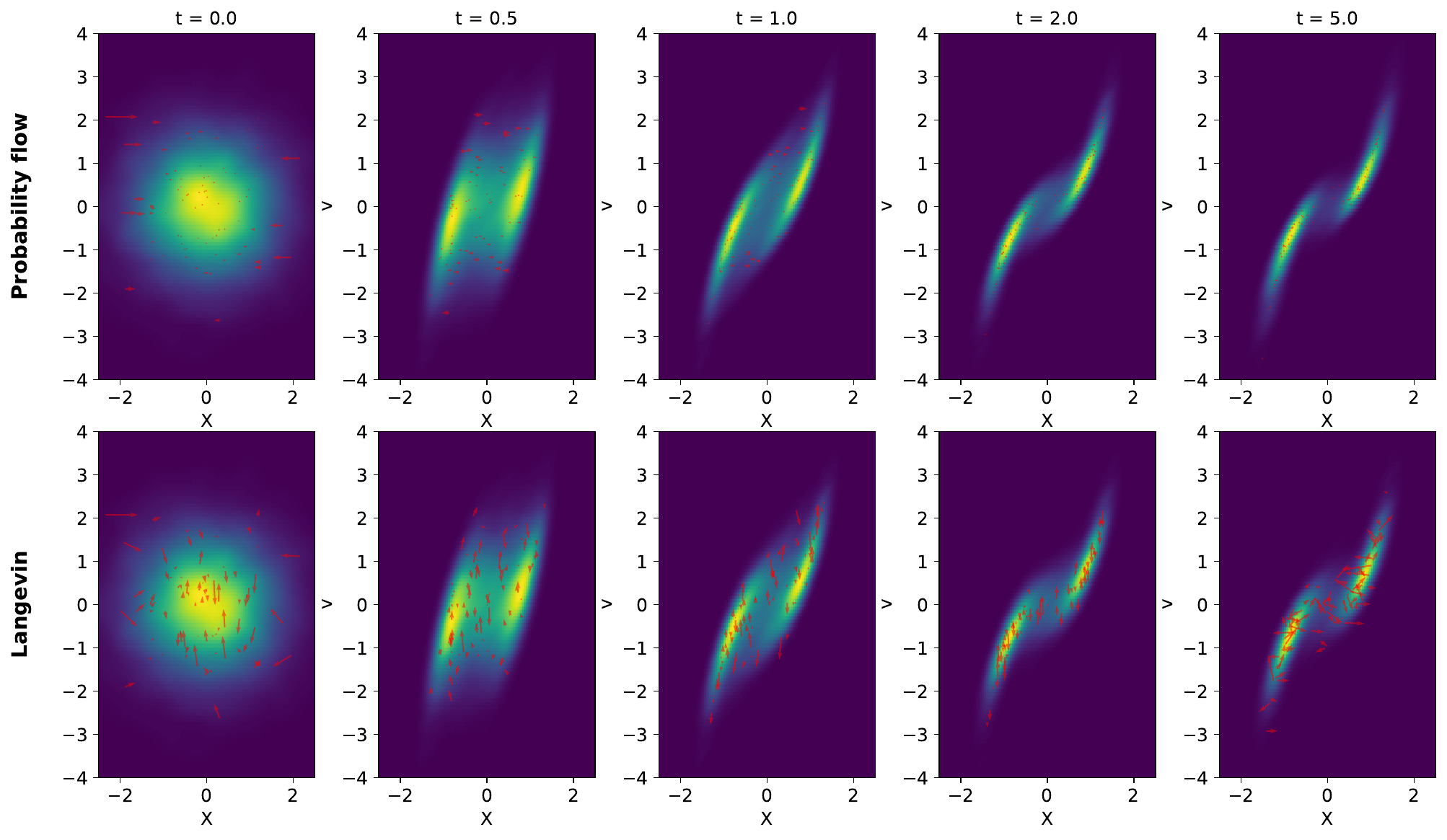}
\caption{Numerical results for the active swimmer model. Top row: deterministic probability flow generated by the learned velocity field. Bottom row: reference stochastic dynamics \eqref{eq:Swimmer_stochastic}.}
\label{fig:Swimmer}
\end{figure}

Figure \ref{fig:Swimmer} shows the numerical results with $\gamma=0.1$, $\ve=1.0$. The training time is $2861$ seconds. The visualization follows the same format as that used for the van der Pol oscillator. The learned deterministic probability flow accurately reproduces the evolution of the stochastic dynamics through the composed velocity field, without explicitly simulating the Brownian noise.

\subsection{Scalability test in high dimensions}
In this subsection, we investigate the performance of the proposed method on high-dimensional Fokker--Planck equations. We first revisit the high-dimensional examples considered in the previous sections to demonstrate the scalability of our algorithm. We then consider the 100-dimensional harmonically interacting particle benchmark studied in \cite{boffi2023probability}.

\subsubsection{Previous high-dimensional examples}

We test our method on $50$-dimensional overdamped Langevin dynamics and $50$-dimensional underdamped Langevin dynamics, with both quadratic and double-well potentials. These experiments are designed to probe the performance of the proposed algorithm in regimes where grid-based discretizations are infeasible due to the curse of dimensionality.

\paragraph{Overdamped Langevin dynamics with $d=50$}
For the quadratic potential, the training time is $31437$ seconds; 
for the double-well potential, the training time is $51822$ seconds. 
For the quadratic case, the errors are
\[
\text{err}_{f} = \num{3.53e-2}, 
\qquad 
\text{err}_{s} = \num{5.30e-2}.
\]
Due to the curse of dimensionality, the density $\rho$ is exponentially small in most regions of the state space, so $\text{err}_{\rho}$ becomes numerically 
ill-conditioned and is therefore not reported in this regime. For the double-well case, we evaluate performance through projected particle distributions, 
as described below.

Figure~\ref{fig:CG50d} shows the numerical results for the quadratic potential. To visualize the learned velocity field in high dimension, we plot one-dimensional slices: the $1$st and $26$th coordinates of the velocity field at $t=0$, where one coordinate of $x$ varies while all other coordinates are fixed at zero. The first row demonstrates that the neural network accurately approximates the composed velocity field. The second row shows the evolution of the free energy and its dissipation, confirming that the learned probability flow preserves the correct thermodynamic structure.
\begin{figure}[!ht]
\centering
\includegraphics[width=0.6\linewidth]{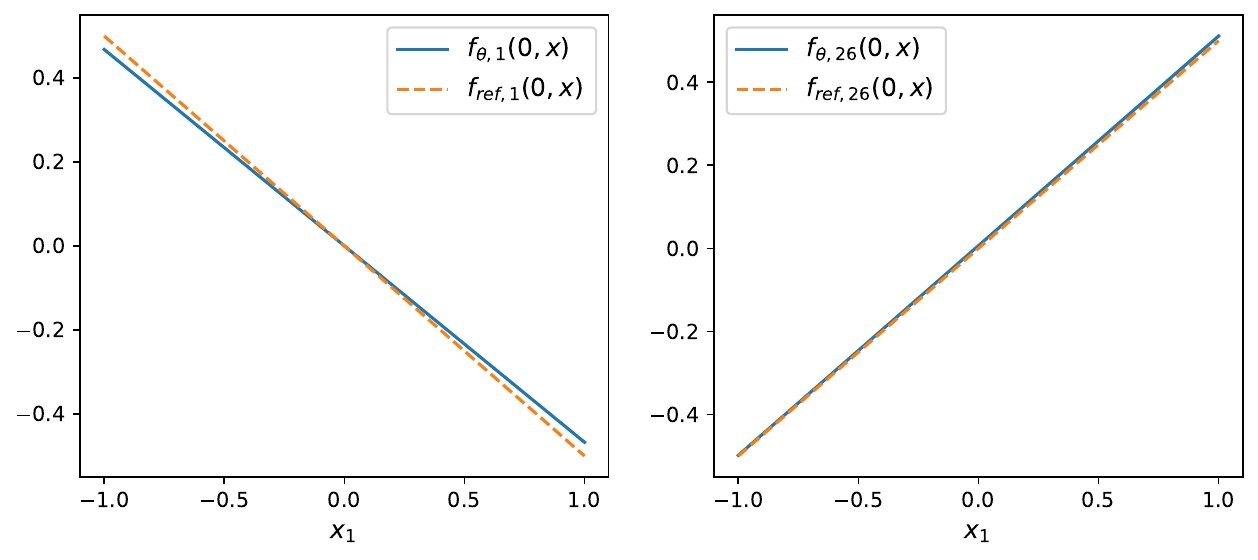}
\includegraphics[width=0.6\linewidth]{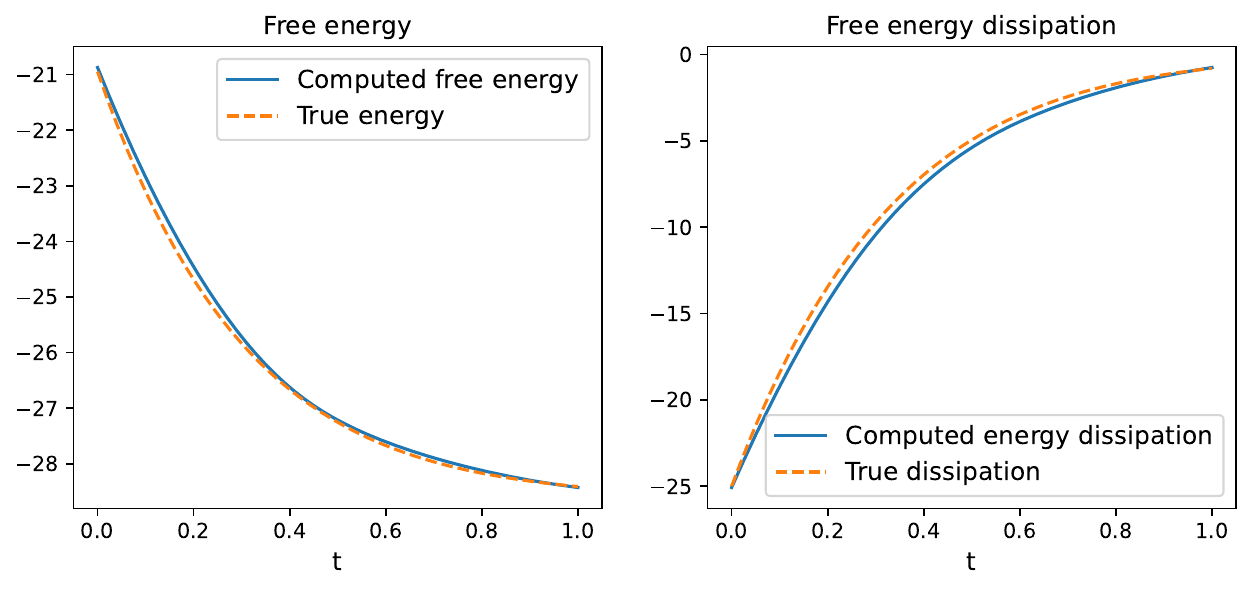}
\caption{Overdamped Langevin dynamics in $d=50$ dimensions with quadratic potential. First row: one-dimensional slices of the learned velocity field at $t=0$. Second row: free energy and its dissipation.}
\label{fig:CG50d}
\end{figure}

For the $50$-dimensional double-well potential, the two wells are located at $c_1 = (1,\dots,1)$ and $c_2 = (-1,\dots,-1)$. To visualize the high-dimensional particle distribution, we project $x \in \RR^{50}$ onto a two-dimensional coordinate system aligned with the well axis. Define
$$u=\tfrac{1}{\sqrt{d}}(1,\dots,1), \qquad x_\parallel = u^\top x, \qquad r_\perp = \| x - x_\parallel u \|.$$
We plot the particle distribution in $(x_\parallel, r_\perp)$ space at terminal time, comparing the learned probability flow with reference Langevin dynamics. As shown in Figure~\ref{fig:DW50d}, the learned flow correctly separates particles into the two metastable wells and matches the reference dynamics.
\begin{figure}[!ht]
\centering
\includegraphics[width=0.6\linewidth]{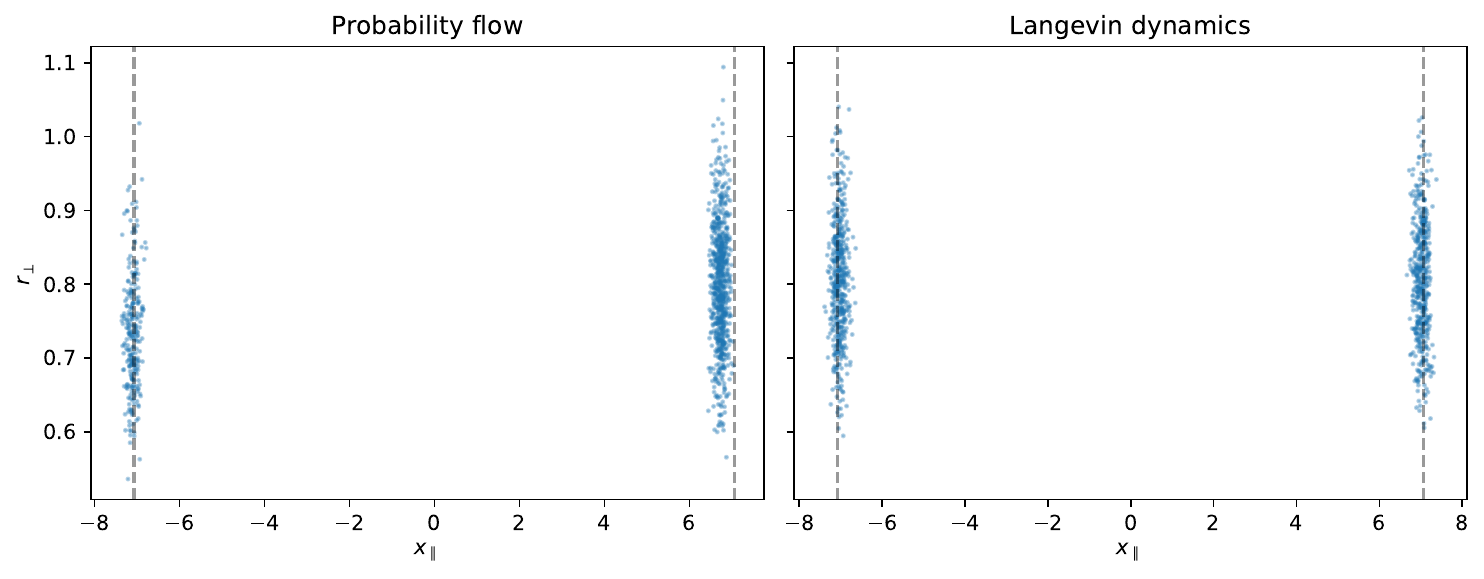}
\caption{Overdamped Langevin dynamics in $d=50$ dimensions with double-well potential. Projected particle distribution in $(x_\parallel, r_\perp)$ coordinates at terminal time.}
\label{fig:DW50d}
\end{figure}

\paragraph{Underdamped Langevin dynamics with $2d=50$}
We next consider underdamped Langevin dynamics with quadratic and double-well potentials, where $(x,v) \in \RR^{25} \times \RR^{25}$ so that the total phase-space dimension is $50$. 
The training times for the quadratic and double-well cases are $24358$ seconds and $26479$ seconds, respectively.
For the quadratic potential, the errors are
\[\text{err}_{f} = \num{8.50e-2}, \qquad \text{err}_{s} = \num{1.01e-1}.\]
Figure~\ref{fig:ULDGaussain50d} shows one-dimensional slices of the learned velocity component $f_v$ at $t=0$: the first coordinate of $f_v$ is plotted as a function of $x_1$ (left) and $v_1$ (right), with all other coordinates fixed at zero. The second row presents the free energy and dissipation, demonstrating that the method accurately captures the thermodynamic structure in high-dimensional phase space.
\begin{figure}[!ht]
\centering
\includegraphics[width=0.6\linewidth]{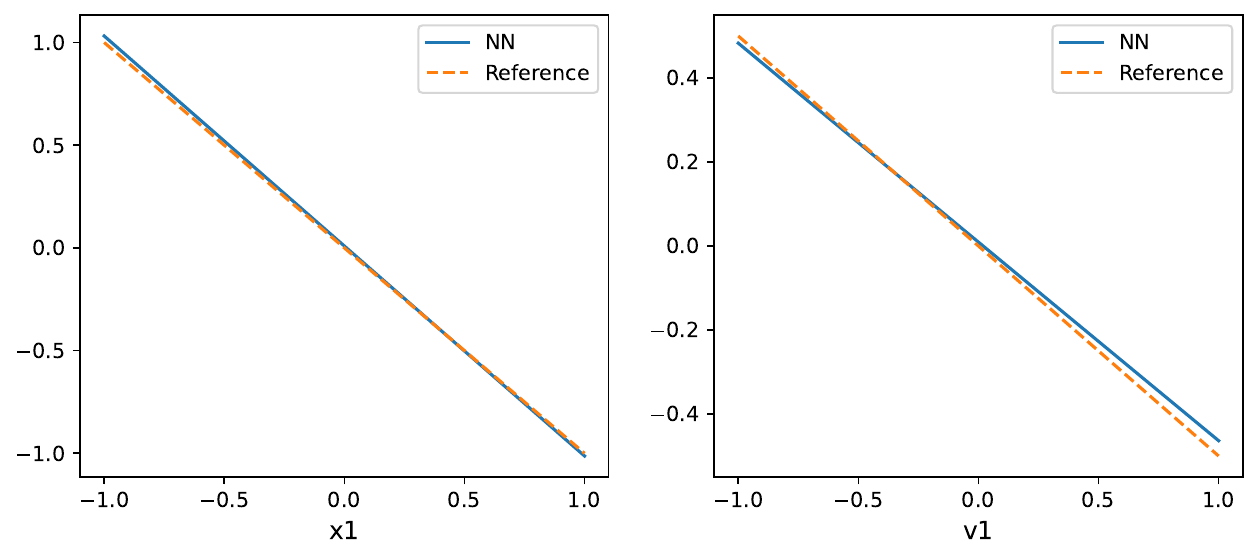}
\includegraphics[width=0.6\linewidth]{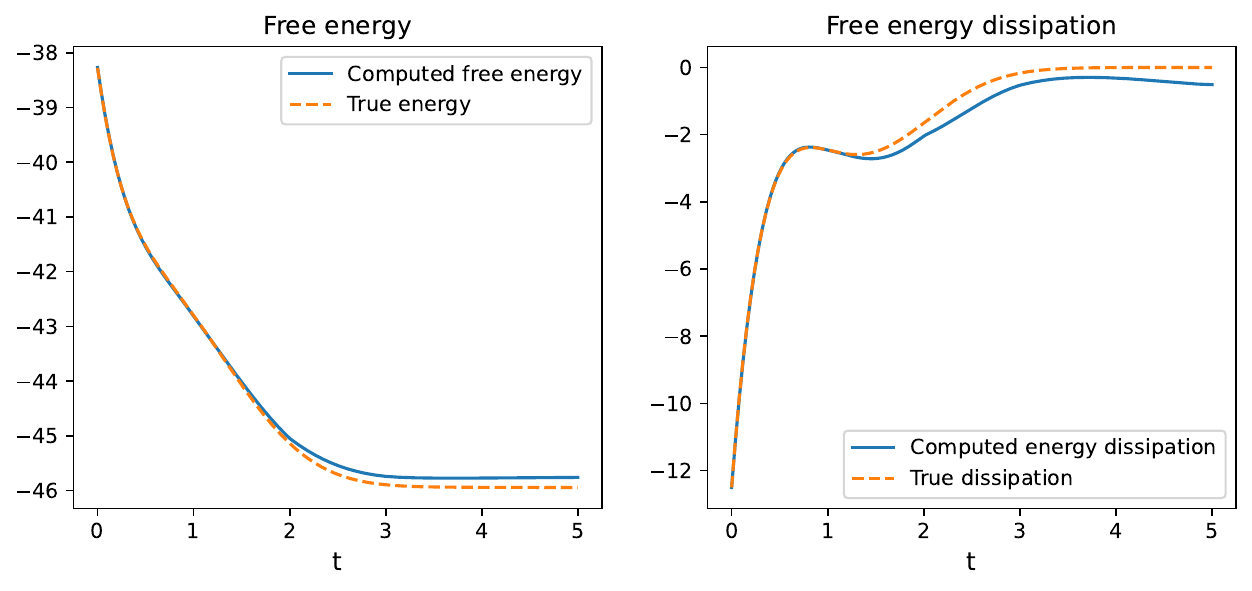}
\caption{Underdamped Langevin dynamics in $2d=50$ dimensions with quadratic potential. First row: slices of the learned velocity field $f_v$ at $t=0$. Second row: free energy and its dissipation.}
\label{fig:ULDGaussain50d}
\end{figure}

For the double-well potential, we again project the terminal distribution of $x_T \in \RR^{25}$ onto $(x_\parallel, r_\perp)$ coordinates as in the overdamped case. Figure~\ref{fig:ULDDW50d} shows that the learned probability flow reproduces the correct bimodal structure and agrees well with the reference underdamped Langevin dynamics. The projected distribution appears less sparse than in the $d=50$ overdamped case, since the configuration space dimension is $25$, mitigating the severity of the 
curse of dimensionality.
\begin{figure}[!ht]
\centering
\includegraphics[width=0.6\linewidth]{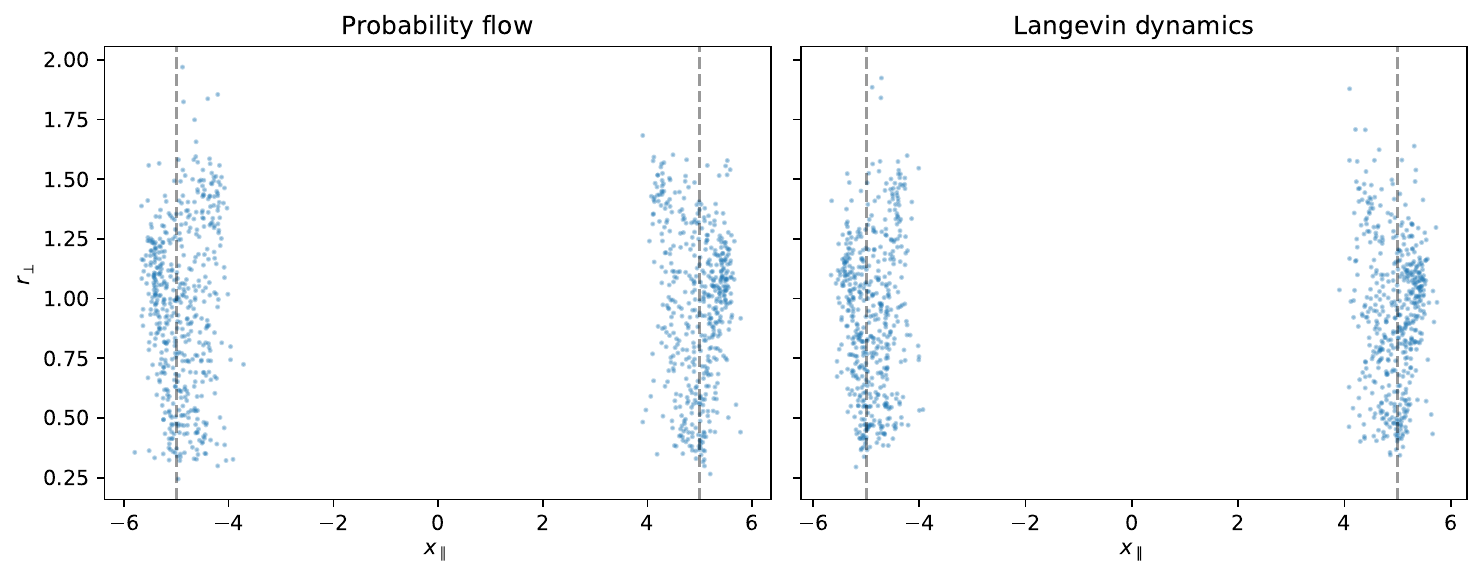}
\caption{Underdamped Langevin dynamics in $2d=50$ dimensions with double-well potential. Projected terminal particle distribution of the probability flow and the reference dynamics.}
\label{fig:ULDDW50d}
\end{figure}

\subsubsection{Harmonically interacting particles}
In this section, we consider the harmonically interacting particle system studied in \cite{boffi2023probability}. The stochastic dynamics of $N$ interacting particles in $\RR^2$ are given by
\begin{equation}
    \rd x_t^{(i)}
    = \Big[\left(\beta_t-x_t^{(i)}\right)
    + \alpha\Big(x_t^{(i)}-\frac1N\sum_{j=1}^N x_t^{(j)}\Big)\Big] \,\rd t
    + \sqrt{2\varepsilon} \,\rd W_t^{(i)},
    \qquad i=1,\ldots,N.
    \label{eq:SDE_particle}
\end{equation}
Here, $\beta_t = (\cos(\pi\omega t), \sin(\pi\omega t))\tp$ is the trap. We use the parameters $N=50,a=2,\omega=1,\alpha=0.5,\varepsilon=0.25$ and the initial points are i.i.d. sampled $x_0^{(i)}\sim \mathcal N(\beta_0,\sigma_0^2 I_2)$ with $\sigma_0=0.5$.
Concatenating all particle coordinates yields a 100-dimensional FP equation, and the joint distribution remains Gaussian $N(m_t, C_t)$. Here, $m_t$ consists of $N$ copies of $e^{-t}\beta_0 + \int_0^te^{s-t} \beta_s \,\rd s$. $C_t = \Sigma_t \otimes I_2$ and
$$\Sigma_t = \big[(\sigma_0^2-\ve)e^{-2t}+\ve\big] \, \frac{1}{N} \mathbf{1}\mathbf{1}\tp + \Big[\Big(\sigma_0^2-\frac{\ve}{1-\alpha}\Big) e^{-2(1-\alpha)t} + \frac{\ve}{1-\alpha}\Big] \Big(I_N - \frac{1}{N} \mathbf{1}\mathbf{1}\tp\Big)$$
where $\mathbf{1}\in\RR^N$ is the all-one column vector. Therefore, the analytical solution is available in closed form.

The training time is $88560$ seconds. Our numerical results are comparable to \cite{boffi2023probability}. The corresponding absolute errors are $\text{err}_{f} = \num{1.16e-1}$, $\text{err}_{s} = \num{5.52e-1}$. For comparison with \cite{boffi2023probability}, we additionally report the relative errors of the velocity field and the score function,
\[
\text{err}_{\mathrm{rel},f}=\num{1.38e-3},
\qquad
\text{err}_{\mathrm{rel},s}=\num{6.61e-3},
\]
where the relative errors are defined by
\begin{equation*}
\begin{aligned}
\text{err}_{\mathrm{rel},f} &= \sum_{n=1}^{N_x} \sum_{j=0}^{N_t} \abs{f\parentheses{t_j, x^{(n)}_{t_j};\theta} - f\parentheses{t_j, x^{(n)}_{t_j}}}^2 \Big/ \,\sum_{n=1}^{N_x} \sum_{j=0}^{N_t} \abs{f\parentheses{t_j, x^{(n)}_{t_j}}}^2,\\
\text{err}_{\mathrm{rel},s} &=  \sum_{n=1}^{N_x}\sum_{j=1}^{N_t}  \abs{ s^{(n)}_{t_{j}} - \nx \log \rho\parentheses{t_{j}, x^{(n)}_{t_{j}}}}^2 \Big/ \,\sum_{n=1}^{N_x}\sum_{j=1}^{N_t}  \abs{\nx \log \rho\parentheses{t_{j}, x^{(n)}_{t_{j}}}}^2.
\end{aligned}
\end{equation*}

\begin{figure}[!ht]
\centering
\includegraphics[width=0.8\linewidth]{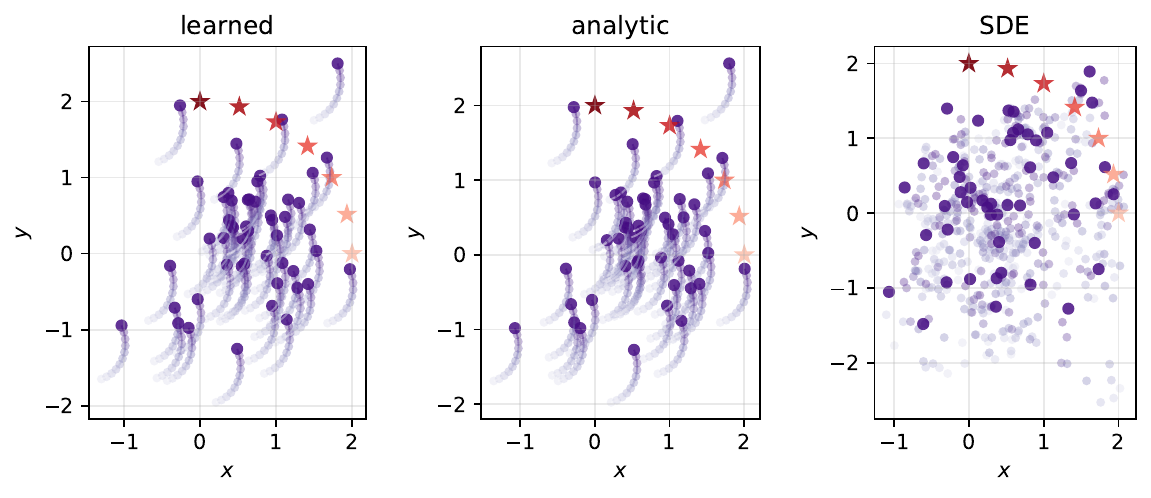}
\includegraphics[width=0.8\linewidth]{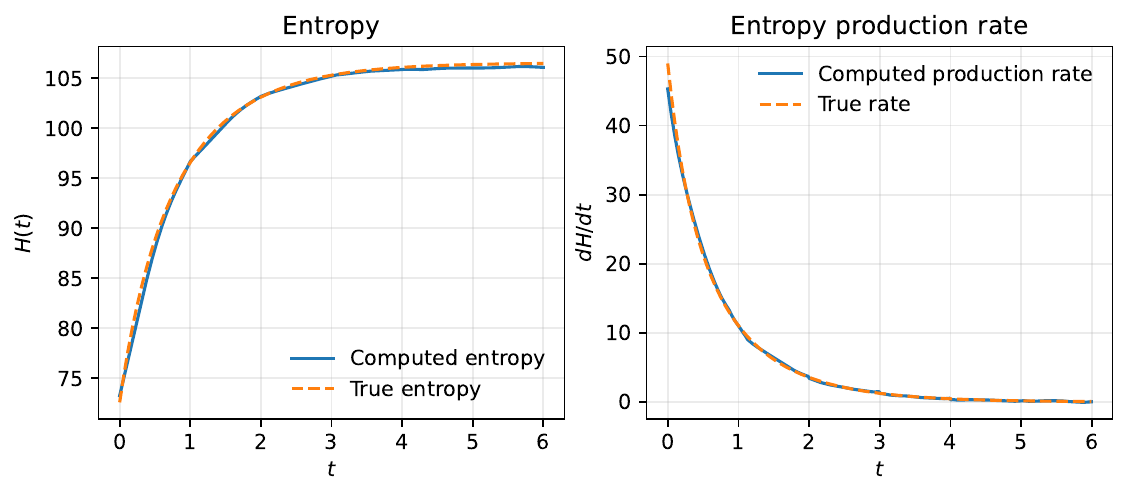}
\caption{$100$ dimensional FP equation for harmonically interacting particles. First row: the learned, analytic, and stochastic particles trajectories for $t \in [2.0, 2.5]$. Second row: the learned and analytical entropy and its production rate.}
\label{fig:Harmonic100d}
\end{figure}

Figure~\ref{fig:Harmonic100d} (top row) compares the trajectories of representative two-dimensional particles generated by the learned probability flow, the analytical probability flow, and the reference SDE. The red stars denote the trajectory of the trap $\beta_t$. The learned trajectories closely match the analytical probability flow and exhibit smooth deterministic evolution, in contrast to the stochastic trajectories generated by the SDE. 
The second row compares the computed entropy $H_t = - \int_{\RR^d} \log\rho(t,x) \, \rho(t,x) \, \rd x$ and its production rate $\pt H_t$ with their analytical counterparts. In the numerical experiments, we estimate the entropy and its production rate through
$$\widehat H(t_j) = -\frac1{N_x}\sum_{n=1}^{N_x}\ell_{t_j}^{(n)}, \qquad \widehat{\pt H_t}(t_j) =\frac1{N_x}\sum_{n=1}^{N_x} \nabla_x\cdot f_\theta(t_j,x_{t_j}^{(n)}),$$
where the second identity follows from
$$\pt H_t = - \pt \,\EE[\log\rho(t,x_t)] = -\EE[\pt \log\rho(t,x_t)] = \EE[\nx \cdot f(t,x_t)].$$
The second row of Figure~\ref{fig:Harmonic100d} demonstrates that the proposed score-based probability flow accurately reproduces both the entropy and its production rate.

\begin{figure}[!ht]
\centering
\includegraphics[width=\linewidth]{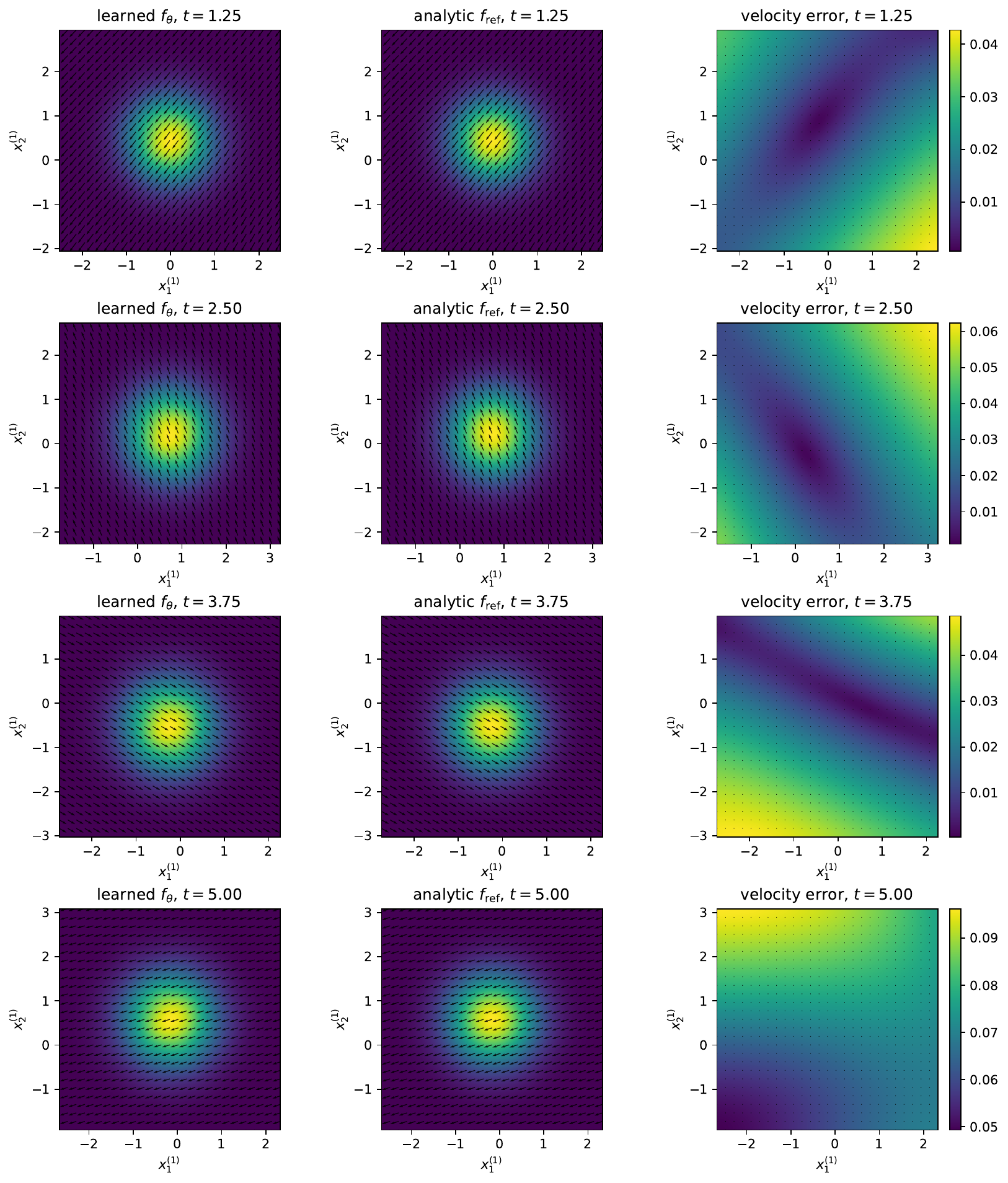}
\caption{First column: the learned velocity field $f$ and the marginal density estimation for $x^{(1)}_t$ at $t=1.25,2.50,3.75,5.00$. Second column: the analytical velocity field and density function. Third column: the error vector field and its magnitude.}
\label{fig:Harmonic_velocity}
\end{figure}

Figure~\ref{fig:Harmonic_velocity} compares the learned and analytical velocity fields at $t=1.25, 2.50, 3.75, 5.00$. The first column shows the learned velocity field of the first particle $x^{(1)}_t$, together with its marginal density in the background. Each velocity field is plotted as a function of a single particle’s coordinate (denoted as $x^{(1)}_1$ and $x^{(1)}_2$), while all other particles are fixed at their mean values given by $m_t$. The background color represents the marginal density, estimated by kernel density estimation from samples generated by the learned probability flow. The second column shows the analytical reference solution. The third column visualizes the error vector field, with the background color indicating its magnitude. The learned velocity field agrees closely with the analytical composed velocity field across all four time instances, with only small approximation errors.

We remark that, unlike \cite{boffi2023probability}, our algorithm optimizes the entire trajectory rather than minimizing the local truncation error at each time step. On this benchmark, our method achieves higher accuracy than \cite{boffi2023probability}, while requiring higher computational cost.

\section{Convergence analysis}\label{sec:convergence}
The numerical experiments in Section \ref{sec:results} demonstrate that the proposed MFC formulation combined with score-based normalizing flows can accurately approximate the density evolution and free-energy dissipation. In this section, we complement these empirical results with a theoretical analysis in a tractable setting. We isolate the core optimization structure induced by the MFC velocity-matching formulation, and leave the analysis for the full nonlinear neural-network parameterization for future research.
To this end, we present a convergence analysis for the flow matching problem in the canonical case of the Ornstein--Uhlenbeck (OU) process. For convergence analysis under more general settings, we refer the readers to \cite{han2020convergence,gu2021mean,carmona2021convergence,carmona2022convergence,zhou2023policy,sethi2024entropy,zhou2024solving,ma2024convergence}. Consider the OU process in $\RR^d$
$$\rd x_t = b(x_t) \,\rd t + \sqrt{2\gamma}\,\rd W_t$$
with $x_0 \sim N(\mu_0, \Sigma_0)$. Here $b(x) = B_1 x + b_0$, $B_1 \in \RR^{d\times d}$, $b_0 \in \RR^d$. The FP equation is 
$$0 = \pt \rho + \nx\cdot(\rho \, b(x)) - \gamma \Delta_x \rho = \pt \rho + \nx\cdot( \rho (b(x) - \gamma \nx \log\rho) ).$$
We consider the linear parametrization $f(t,x;\theta) = \Theta_1(t) x + \theta_0(t)$, where $\Theta_1: [0,T] \to \RR^{d\times d}$, $\theta_0: [0,T] \to \RR^d$. We further discretize it into a one-step flow matching problem with forward Euler scheme. The parameter is $\theta = (\theta_0, \Theta_1) \in \RR^d \times \RR^{d \times d}$. The population loss (in expectation) is
\begin{equation}\label{eq:population_loss}
\begin{aligned}
L(\theta) &= \EE_{x\sim N(\mu_0,\Sigma_0)} \sqbra{\abs{f(x;\theta) - b(x) + \gamma s(x)}^2} \\
& = \EE_{x\sim N(\mu_0,\Sigma_0)} \sqbra{\abs{\Theta_1 x + \theta_0 - (B_1 x + b_0) - \gamma \Sigma_0^{-1}(x-\mu_0)}^2},
\end{aligned}
\end{equation}
where we omit $\dt=T$ in \eqref{eq:loss_discrete}. Here, $s(x) = \nx \log\rho(0,x)=-\Sigma_0^{-1}(x-\mu_0)$ is the score function at $t=0$. With $N_x$ samples $\{x^{(n)}\}_{n=1}^{N_x}$ drawn from $\rho_0$, the empirical loss (with finite samples) is
\begin{equation*}
\wh{L}\parentheses{\theta, \{x^{(n)}\}_{n=1}^{N_x}} = \dfrac{1}{N_x} \sum_{n=1}^{N_x} \abs{\Theta_1 x^{(n)} + \theta_0 - (B_1 x^{(n)} + b_0) - \gamma \Sigma_0^{-1}(x^{(n)}-\mu_0)}^2.
\end{equation*}
We denote $\sigma_0$ is smallest eigenvalue of $\Sigma_0$, and
\begin{equation}\label{eq:lam_0}
\lam_0 = \frac12\parentheses{ 1 + \sigma_0 + |\mu_0|^2 - \sqrt{\parentheses{1 + \sigma_0 + |\mu_0|^2}^2 -4\sigma_0}}.
\end{equation}
Then, the gradient descent algorithm
$$\theta^{(k+1)} = \theta^{(k)} - \eta \nt \wh{L}\parentheses{\theta^{(k)}, \{x^{(n)}\}_{n=1}^{N_x}}$$
has the following convergence property.
\begin{theorem}\label{thm:convergnce}
Let $\ve > 0$ and $\delta \in (0,1)$. Let the number of samples $N_x = \Omega( \lam_0^{-2} (d - \log(\delta)))$. Assume the step size for gradient descent satisfies $$\eta \le \frac43/(8(1+|\mu_0|^2)+4\norm{\Sigma_0}_2 + \lam_0).$$ Then, the gradient descent method on the empirical loss satisfies
$$\mathrm{Pr}\parentheses{\wh{L}\parentheses{\theta^{(K)}, \{x^{(n)}\}_{n=1}^{N_x}} \le \ve} \ge 1 - \delta$$ 
after $K = \Omega(\log(\ve^{-1})(\eta\lam_0)^{-1})$ steps.
\end{theorem}
\begin{proof}
\noindent\emph{Step 1.} We characterize the landscape of the population loss \eqref{eq:population_loss}. We take derivative of the population loss and obtain the critical point equations
\begin{align*}
0 &= \partial_{\theta_0} L(\theta) = 2\EE_{x\sim N(\mu_0,\Sigma_0)} \sqbra{\Theta_1 x + \theta_0 - (B_1 x + b_0) - \gamma \Sigma_0^{-1}(x-\mu_0)} \\
0 &= \partial_{\Theta_1} L(\theta) = 2\EE_{x\sim N(\mu_0,\Sigma_0)} \sqbra{\parentheses{\Theta_1 x + \theta_0 - (B_1 x + b_0) - \gamma \Sigma_0^{-1}(x-\mu_0)} x\tp}.
\end{align*}
The solution to this critical point system is the unique minimizer of $L(\theta)$, given by
$$\theta_0^* = b_0 - \gamma \Sigma_0^{-1} \mu_0 \quad \text{and} \quad \Theta_1^* = B_1 + \gamma \Sigma_0^{-1}.$$
Let $\theta_{0,j}$ be the $j$-th element of $\theta_0$ and $\Theta_{1,j}$ be the $j$-th row of $\Theta_1$. Then the population loss can be written as $L(\theta) = \sum_{j=1}^d L_j(\theta)$, where
$$L_j(\theta) = \EE_{x\sim N(\mu_0,\Sigma_0)} \sqbra{ \parentheses{\Theta_{1,j} x + \theta_{0,j} - (B_{1,j} x + b_{0,j}) - \gamma (\Sigma_0^{-1})_{j,:}(x-\mu_0)}^2}$$
only depends on $\theta_{0,j}$ and $\Theta_{1,j}$. Therefore, in order to study the optimization landscape for $L(\theta)$, it is sufficient to study the optimization landscape for $L_j(\theta_{0,j}, \Theta_{1,j})$. Its critical point equations are
\begin{align*}
0 &= \partial_{\theta_{0,j}} L_j = \EE_{x\sim N(\mu_0,\Sigma_0)} \sqbra{\Theta_{1,j} x + \theta_{0,j} - (B_{1,j} x + b_{0,j}) - \gamma (\Sigma_0^{-1})_{j,:}(x-\mu_0)} \\
&= \Theta_{1,j} \mu_0 + \theta_{0,j} - (B_{1,j} \mu_0 + b_{0,j}),
\end{align*}
and
\begin{align*}
0 &= \partial_{\Theta_{1,j}} L_j = \EE_{x\sim N(\mu_0,\Sigma_0)} \sqbra{ \parentheses{\Theta_{1,j} x + \theta_{0,j} - (B_{1,j} x + b_{0,j}) - \gamma (\Sigma_0^{-1})_{j,:}(x-\mu_0)}x\tp} \\
& = \Theta_{1,j} (\Sigma_0 + \mu_0\mu_0\tp) + \theta_{0,j} \mu_0\tp - (B_{1,j} (\Sigma_0 + \mu_0\mu_0\tp) + b_{0,j}\mu_0\tp) - \gamma e_j\tp,
\end{align*}
where $e_j$ is the $j$-th standard unit vector. The Hessian of $L_j$ is
\begin{equation}\label{eq:Hessian_Lj}
\nabla_{\theta_{0,j}, \Theta_{1,j}}^2 L_j(\theta_{0,j}, \Theta_{1,j}) = 2\begin{bmatrix} 
1 & \mu_0\tp \\ \mu_0 & \Sigma_0 + \mu_0\mu_0\tp \end{bmatrix}.
\end{equation}

\noindent\emph{Step 2.} We estimate the Hessian for the population loss. First, we can verify through definition \eqref{eq:lam_0} that $\lam_0 \in (0,1]$. We claim that the smallest eigenvalue of the Hessian \eqref{eq:Hessian_Lj} is larger than or equal to $\lam_0$. If $\mu_0=0$, then $\lam_0 = \min(1,\sigma_0)$, and claim is clear. If $\mu_0\neq0$, then $\lam_0 \in (0,1)$. For any $a\in \RR$ and $b \in \RR^d$, we have
\begin{align*}
& \quad \frac12 [a,b\tp] \, \nabla_{\theta_{0,j}, \Theta_{1,j}}^2 L_j(\theta_{0,j}, \Theta_{1,j}) \begin{bmatrix}a\\b\end{bmatrix} = [a,b\tp]\begin{bmatrix} 
1 & \mu_0\tp \\ \mu_0 & \Sigma_0 + \mu_0\mu_0\tp \end{bmatrix}\begin{bmatrix}a\\b\end{bmatrix} \\
& = a^2 + 2 ab\tp \mu_0 + b\tp (\Sigma_0 + \mu_0\mu_0\tp) b \ge a^2 + 2 ab\tp \mu_0 + b\tp (\sigma_0 I_d + \mu_0\mu_0\tp) b\\
&= \sqbra{(1-\lam_0) a^2 + 2 ab\tp \mu_0 + \dfrac{1}{1-\lam_0} b\tp\mu_0\mu_0\tp b} + \lam_0 a^2 + b\tp \parentheses{\sigma_0 I_d - \dfrac{\lam_0}{1-\lam_0}\mu_0\mu_0\tp}b\\
& \ge \dfrac{1}{1-\lam_0}\parentheses{(1-\lam_0)a + \mu_0\tp b}^2 + \lam_0 a^2 + \parentheses{\sigma_0 - \dfrac{\lam_0}{1-\lam_0} |\mu_0|^2} |b|^2 \ge \lam_0 (a^2 + |b|^2),
\end{align*}
where the last inequality in because $\lam_0 = \sigma_0 - \frac{\lam_0}{1-\lam_0} |\mu_0|^2$. Therefore, $L_j$ is $2\lam_0$-strongly convex in $(\theta_{0,j}, \Theta_{1,j})$, which implies $L(\theta)$ is $2\lam_0$-strongly convex in $(\theta_0,\Theta_1)$. We also have the upper bound
\begin{align*}
& \quad \frac12 [a,b\tp] \, \nabla_{\theta_{0,j}, \Theta_{1,j}}^2 L_j(\theta_{0,j}, \Theta_{1,j}) \begin{bmatrix}a\\b\end{bmatrix}  \\
& = a^2 + 2 ab\tp \mu_0 + b\tp (\Sigma_0 + \mu_0\mu_0\tp) b \\
& \le 2a^2 + (\norm{\Sigma_0}_2 + 2|\mu_0|^2) |b|^2,
\end{align*}
which implies $\norm{\nabla_{\theta_{0,j}, \Theta_{1,j}}^2 L_j}_2 \le 4(1 + |\mu_0|^2) + 2\norm{\Sigma_0}_2$ and $\norm{\nabla_\theta^2 L(\theta)}_2 \le 4(1 + |\mu_0|^2) + 2\norm{\Sigma_0}_2$.

\noindent\emph{Step 3.} We study the optimization landscape for the empirical loss $\wh{L}(\theta, \{x^{(n)}\}_{n=1}^{N_x})$. We denote $\hm := \frac{1}{N_x} \sum\limits_{n=1}^{N_x} x^{(n)}$ and $\hD := \frac{1}{N_x} \sum\limits_{n=1}^{N_x} x^{(n)}x^{(n)\top}$ the empirical mean and second order moments. Similar to \emph{step 2}, we can decompose the loss into $\wh{L}(\theta, \{x^{(n)}\}_{n=1}^{N_x}) = \sum\limits_{j=1}^d \wh{L}_j(\theta, \{x^{(n)}\}_{n=1}^{N_x})$, where
$$\wh{L}_j(\theta,\{x^{(n)}\}_{n=1}^{N_x}) = \frac{1}{N_x} \sum_{n=1}^{N_x} \sqbra{ \parentheses{\Theta_{1,j} x^{(n)} + \theta_{0,j} - (B_{1,j} x^{(n)} + b_{0,j}) - \gamma (\Sigma_0^{-1})_{j,:}(x^{(n)}-\mu_0)}^2},$$
only depends on $\theta_{0,j}$ and $\Theta_{1,j}$. Taking derivatives, we obtain the critical point equations
\begin{align*}
0 &= \partial_{\theta_{0,j}} \wh{L}_j = \frac{2}{N_x} \sum_{n=1}^{N_x} \sqbra{ \Theta_{1,j} x^{(n)} + \theta_{0,j} - (B_{1,j} x^{(n)} + b_{0,j}) - \gamma (\Sigma_0^{-1})_{j,:}(x^{(n)}-\mu_0)} \\
&= \Theta_{1,j} \hm + \theta_{0,j} - (B_{1,j} \hm + b_{0,j})- \gamma (\Sigma_0^{-1})_{j,:}(\hm-\mu_0),
\end{align*}
and
\begin{align*}
0 &= \partial_{\Theta_{1,j}} \wh{L}_j =\frac{2}{N_x} \sum_{n=1}^{N_x} \sqbra{ \Theta_{1,j} x^{(n)} + \theta_{0,j} - (B_{1,j} x^{(n)} + b_{0,j}) - \gamma (\Sigma_0^{-1})_{j,:}(x^{(n)}-\mu_0)}x\tp_n \\
& = \Theta_{1,j} \hD + \theta_{0,j} \hm \tp - (B_{1,j} \hD + b_{0,j}\hm\tp) - \gamma (\Sigma_0^{-1})_{j,:}(\hD-\mu_0\hm\tp).
\end{align*}
The Hessian of $L_j$ is
\begin{equation}\label{eq:Hessian_Lj2}
\nabla_{\theta_{0,j}, \Theta_{1,j}}^2 \wh{L}_j(\theta_{0,j}, \Theta_{1,j}, \{x^{(n)}\}_{n=1}^{N_x}) = 2\begin{bmatrix} 
1 & \hm\tp \\ \hm & \hD \end{bmatrix}.
\end{equation}

We observe that the critical point system has a unique solution $\theta_{0,j}^* = b_{0,j} - \gamma (\Sigma_0^{-1})_{j,:} \mu_0$, $\Theta_{1,j}^* = B_{1,j} + \gamma (\Sigma_0^{-1})_{j,:}$ if and only if 
the Hessian \eqref{eq:Hessian_Lj2} is invertible. Let
$$\hS = \hD - \hm\hm\tp$$
be the empirical estimation for the covariance matrix. Then the invertibility of the Hessian \eqref{eq:Hessian_Lj2} is equivalent to invertibility of $\hS$. When $N \ge d$, $\hS$ is invertible almost surely. In this case, we observe that the minimizer for the empirical loss is the same as the population loss.

Next, we apply the concentration theorem \cite[Theorem 4.6.1]{vershynin2018high}. Let $N_x = \Omega( \lam_0^{-2} (d - \log(\delta)))$. Then, with probability $\ge 1-\delta$, we have
\begin{align*}
& \quad \norm{\nt^2 \wh{L} - \nt^2 L}_2=\norm{ 2\begin{bmatrix} 
1 & \hm\tp \\ \hm & \hD \end{bmatrix} - 2\begin{bmatrix} 
1 & \mu_0\tp \\ \mu_0 & \Sigma_0 + \mu_0\mu_0\tp \end{bmatrix} }_2 \\
& \le C \dfrac{1}{\sqrt{N_x}} \parentheses{\sqrt{d} + \sqrt{-\log\delta}} \le \frac12 \lam_0,
\end{align*}
where the last inequality is because $N_x = \Omega( \lam_0^{-2} (d - \log(\delta)))$. Combining with the estimates in \emph{step 2}, we obtain that 
$$\nt^2 \wh{L} \ge \frac32 \lam_0 I \quad \text{and} \quad \norm{\nt^2 \wh{L}} \le \frac12 \lam_0 + 4(1 + |\mu_0|^2) + 2\norm{\Sigma_0}_2.$$
The first inequality implies that 
$$\abs{\nt\wh{L}(\theta)}^2 \ge 3\lam_0 \parentheses{ \wh{L}(\theta) - \wh{L}(\theta^*)} = 3\lam_0 \wh{L}(\theta),$$
where we omit the input $\{x^{(n)}\}_{n=1}^{N_x}$ in $\wh{L}$.
Therefore, under the gradient descent algorithm
$$\theta^{(k+1)} = \theta^{(k)} - \eta \nt \wh{L}\parentheses{\theta^{(k)}, \{x^{(n)}\}_{n=1}^{N_x}},$$
the loss function satisfies
\begin{align*}
& \quad \wh{L}\parentheses{\theta^{(k+1)}} = \wh{L}\parentheses{\theta^{(k)} - \eta \nt L\parentheses{\theta^{(k)}}} \\
& = \wh{L}\parentheses{\theta^{(k)}} - \eta \nt L\parentheses{\theta^{(k)}}\tp \nt L\parentheses{\theta^{(k)}} + \frac12 \eta^2 \nt L\parentheses{\theta^{(k)}}\tp \nt^2 L\parentheses{\xi^{(k)}} \nt L\parentheses{\theta^{(k)}} \\
& \le \wh{L}\parentheses{\theta^{(k)}} - \eta \abs{\nt L\parentheses{\theta^{(k)}}}^2 + \frac12 \eta^2 \parentheses{\frac12 \lam_0 + 4(1 + |\mu_0|^2) + 2\norm{\Sigma_0}_2} \abs{\nt L\parentheses{\theta^{(k)}}}^2 \\
& \le \wh{L}\parentheses{\theta^{(k)}} - \frac23 \eta \abs{\nt L\parentheses{\theta^{(k)}}}^2 \le (1 - 2 \eta \lam_0) \, \wh{L}\parentheses{\theta^{(k)}}.
\end{align*}
Therefore,
$$\wh{L}\parentheses{\theta^{(K)}, \{x^{(n)}\}_{n=1}^{N_x}} \le (1 - 2 \eta \lam_0)^K \, \wh{L}\parentheses{\theta^{(0)}, \{x^{(n)}\}_{n=1}^{N_x}} \le \ve.$$
\end{proof}
We remark that the theorem does not include resampling in each step, because the optimal $\theta$ is the same for the empirical and population losses. We shall consider resampling for a general problem in the future work.

\section{Conclusion and discussions}
\label{sec:conclusion}
In this work, we formulate the flow matching for FP equations as an MFC problem and solve it through the score-based normalizing flow. We conduct a convergence analysis for the flow matching problem of the OU process and validate our algorithms on several examples, including Langevin dynamics, underdamped Langevin dynamics (ULDs), and several chaotic systems.

There are several interesting directions for future work. Firstly, it is worth investigating the impact of alternative time discretization schemes beyond the forward Euler method. We shall analyze how these schemes affect the accuracy and stability of flow-matching in scientific computing and machine learning problems, such as the score-based time-reversible diffusion models. Secondly, while our convergence analysis focuses on the OU process with a short time horizon, extending this analysis to broader classes of dynamics and longer horizons is a meaningful direction. Lastly, this work does not fully explore the dynamics of second-order score functions $H_t = \nx^2\log\rho(t,x_t)$. Understanding and designing fast algorithms for the second-order score function will be left for future study.

\appendix
\section{Derivation of formulas}

\subsection{Explicit formula for dissipation}
In this subsection, we prove the dissipation properties for the free energy stated in the main text. First, we restate and prove Proposition \ref{prop:dissipation_entropy}.
\begin{proposition}[Dissipation of relative entropy]
Let $\rho$ be the solution to the FP equation \eqref{eq:FP_b}, then
\begin{equation*}
\dfrac{\rd}{\rd t} \KL{\rho(t,\cdot)}{\pi} = - \ve \int \abs{\nx \log\dfrac{\rho(t,x)}{\pi(x)}}^2 \rho(t,x) \,\rd x.
\end{equation*}
\end{proposition}
\begin{proof}
We first make an orthogonal decomposition of the flow. We define the flux function as $\gamma(x) := \ve \nx \log \pi(x) - b(x)$, then $\nx \cdot\parentheses{\gamma(x) \pi(x)} = 0$. We can rewrite the FP equation \eqref{eq:FP_b} as
\begin{equation*}
\pt \rho(t,x) = \ve \nx \cdot \parentheses{\rho(t,x) \nx \log \dfrac{\rho(t,x)}{\pi(x)}} + \nx \cdot (\rho(t,x) \gamma(x)).
\end{equation*}
The reason we call it orthogonal decomposition is because
\begin{equation*}
\begin{aligned}
& \quad \int \gamma(x)\tp \nx \log \parentheses{\dfrac{\rho(t,x)}{\pi(x)}} \rho(t,x) \,\rd x = - \int \nx \cdot \parentheses{\gamma(x) \rho(t,x)}  \log \parentheses{\dfrac{\rho(t,x)}{\pi(x)}} \,\rd x \\
& =- \int \nx \cdot \parentheses{\gamma(x) \pi(x) \dfrac{\rho(t,x)}{\pi(x)}}  \log \parentheses{\dfrac{\rho(t,x)}{\pi(x)}} \,\rd x \\
& = - \int \gamma(x)\tp \pi(x) \nx \parentheses{\dfrac{\rho(t,x)}{\pi(x)}} \log \parentheses{\dfrac{\rho(t,x)}{\pi(x)}} \,\rd x \\
& = \int \nx\cdot \parentheses{ \gamma(x) \pi(x) \log \dfrac{\rho(t,x)}{\pi(x)}} \dfrac{\rho(t,x)}{\pi(x)} \,\rd x = \int \gamma(x)\tp \pi(x) \, \nx \log \dfrac{\rho(t,x)}{\pi(x)} \, \dfrac{\rho(t,x)}{\pi(x)} \,\rd x \\
& = \int \gamma(x)\tp \pi(x) \, \nx \parentheses{\dfrac{\rho(t,x)}{\pi(x)}} \,\rd x = - \int \nx \cdot \parentheses{\gamma(x) \pi(x)} \dfrac{\rho(t,x)}{\pi(x)} = 0,
\end{aligned}
\end{equation*}
where we used the fact that $\nx\cdot(\gamma(x) \pi(x)) = 0$. Therefore,
\begin{equation*}
\begin{aligned}
& \quad \dfrac{\rd}{\rd t} \KL{\rho(t,\cdot)}{\pi} = \dfrac{\rd}{\rd t} \int \rho(t,x) \log \dfrac{\rho(t,x)}{\pi(x)} \,\rd x \\
& = \int \pt \rho(t,x) \parentheses{\log \dfrac{\rho(t,x)}{\pi(x)} + 1} \,\rd x = \int \pt \rho(t,x) \log\dfrac{\rho(t,x)}{\pi(x)}  \,\rd x \\
& = \int \sqbra{ \ve \nx \cdot \parentheses{\rho(t,x) \nx \log \dfrac{\rho(t,x)}{\pi(x)}} + \nx \cdot (\rho(t,x) \gamma(x)) } \log\dfrac{\rho(t,x)}{\pi(x)}  \,\rd x \\
& = \int \ve \nx \cdot \parentheses{\rho(t,x) \nx \log \dfrac{\rho(t,x)}{\pi(x)}} \log\dfrac{\rho(t,x)}{\pi(x)} = - \ve \int \abs{\nx \log\dfrac{\rho(t,x)}{\pi(x)}}^2 \rho(t,x) \,\rd x.
\end{aligned}
\end{equation*}
\end{proof}
Next, we restate and prove Proposition \ref{prop:ULD_dissipation}.
\begin{proposition}[Free energy dissipation of ULD]
The ULD satisfies the following energy dissipation formula
\begin{equation*}
\dfrac{\rd}{\rd t} \Dh(\rho(t,\cdot,\cdot)) = - \dfrac{\gamma}{\beta} \int_{\RR^{2d}} \rho(t,x,v) \abs{\nv\log\parentheses{\dfrac{\rho(t,x,v)}{\pi(x,v)}}}^2 \rd x \, \rd v.
\end{equation*}
\end{proposition}
\begin{proof}
Let $\nabla=\nabla_{x,v}$ denote the full gradient. We denote
$$D = \dfrac{\gamma}{\beta} \begin{bmatrix}0&0\\0&I_n\end{bmatrix} \text{ and recall that } J_d = \begin{bmatrix}0&I_n\\-I_n&0\end{bmatrix}.$$
Then, the FP equation \eqref{eq:FP_ULD} can be rewritten as
\begin{align*}
\pt\rho &= \nabla\cdot\sqbra{\rho \parentheses{ \begin{bmatrix} 0\\ \gamma\beta^{-1}(\nv\log\rho + \beta v) \end{bmatrix} - \begin{bmatrix} v\\-\nx U(x) \end{bmatrix}}} \\
& = \nabla\cdot\sqbra{\rho \parentheses{ D\, \nabla\log\dfrac{\rho}{\pi} - J_d \, \nabla\log\pi}} \\
& =  \nabla\cdot\sqbra{\rho \, (D+J_d)\, \nabla\log\dfrac{\rho}{\pi}},
\end{align*}
where the last equality is because $\nabla\cdot(\rho \, J_d \, \nabla\log\rho) = \nabla\cdot(J_d \, \nabla\rho) = 0$. Therefore
\begin{equation*}
\begin{aligned}
& \quad \dfrac{\rd}{\rd t} \Dh(\rho(t,\cdot,\cdot)) = \dfrac{\rd}{\rd t} \KL{\rho(t,\cdot,\cdot)}{\pi} \\
& = \int_{\RR^{2d}} \parentheses{1+\log\parentheses{\dfrac{\rho(t,x,v)}{\pi(x,v)}}} \, \pt \rho(t,x,v) \,\rd x\,\rd v \\
& = \int_{\RR^{2d}} \parentheses{1+\log\parentheses{\dfrac{\rho(t,x,v)}{\pi(x,v)}}} \,\nabla\cdot\sqbra{\rho(t,x,v) (D+J_d)\, \nabla\log\parentheses{ \dfrac{\rho(t,x,v)}{\pi(x,v)}}} \rd x\,\rd v \\
& = - \int_{\RR^{2d}} \nabla\log\parentheses{\dfrac{\rho(t,x,v)}{\pi(x,v)}}\tp \rho(t,x,v) \, (D+J_d)\, \nabla\log\parentheses{\dfrac{\rho(t,x,v)}{\pi(x,v)}} \rd x\,\rd v \\
& = - \int_{\RR^{2d}} \nabla\log\parentheses{\dfrac{\rho(t,x,v)}{\pi(x,v)}}\tp \rho(t,x,v) \, D\, \nabla\log\parentheses{\dfrac{\rho(t,x,v)}{\pi(x,v)}} \rd x\,\rd v \\
& = -\dfrac{\gamma}{\beta} \int_{\RR^{2d}}  \rho(t,x,v) \abs{\nv\log\parentheses{\dfrac{\rho(t,x,v)}{\pi(x,v)}}}^2 \rd x\,\rd v.
\end{aligned}
\end{equation*}
\end{proof}

\subsection{Explicit solution for Gaussian case}
For the Langevin dynamic
$$\rd x_t = -(I_d + cJ_d) x_t \,\rd t + \sqrt{2\ve} \,\rd W_t$$
in Section \ref{sec:LangevinGaussian} with Gaussian initialization $x_0 \sim N(0,\Sigma_0)$, the state remains Gaussian $x_t \sim N(0,\Sigma_t)$. The covariance matrix $\Sigma_t$ satisfies
$$\pt \Sigma_t = A \Sigma_t + \Sigma_t A\tp + 2\ve I_d,$$
where $A := -(I_d + cJ_d)$. If we further assume $\Sigma_0 = \delta I_{2n}$, then
\begin{equation*}
\Sigma(t) = \sqbra{e^{-2t}\delta + \ve(1 - e^{-2t})}I_{2n}.
\end{equation*}
The corresponding score function is 
$$\nx\log\rho(t,x) = -\Sigma(t)^{-1}x = \dfrac{-x}{e^{-2t}\delta + \ve(1 - e^{-2t})}.$$
The true composed velocity is
\begin{equation*}
f(t,x) = b(x) - \ve \nx\log\rho(t,x) = Ax + \dfrac{\ve x}{e^{-2t}\delta + \ve(1 - e^{-2t})}.
\end{equation*}
We use these expressions to provide reference solutions. In the numerical experiments, we set $\delta=1$. Note that $f(t,x)$ does not converge to $0$ as $t\to\infty$ when $c \neq 0$. This makes the Langevin dynamic with non-gradient drift different from gradient drift.

For ULD with quadratic potential function $U(x)=\frac12|x|^2$, we derive the this ODE for $2d=2$ with zero mean $(x_t,v_t) \sim N(0,\Sigma_t)$ for simplicity. A general Gaussian distribution with a non-zero mean can be derived in a similar way. We denote the entries of $\Sigma_t$ by
$$\Sigma_t = \begin{bmatrix} \Sxx_t&\Sxv_t\\\Svx_t&\Svv_t \end{bmatrix}$$
with $\Sxv_t=\Svx_t$. Then we have
$$\pt \Sxx_t = \pt \EE[x_t^2] = 2\EE[x_t v_t] = 2\Sxv_t.$$
Since
$$\rd (x_tv_t) = \sqbra{v_t^2-(\gamma v_t + x_t)x_t} \,\rd t + \sqrt{2\gamma \beta^{-1}} x_t\,\rd W_t,$$
we have
$$\pt \Sxv_t = \Svv_t - \gamma \Sxv_t - \Sxx_t.$$
Also, since
$$\rd(v_t^2) = \sqbra{2\gamma\beta^{-1} - 2v_t(\gamma v_t + x_t)}\,\rd t + 2\sqrt{2\gamma \beta^{-1}} v_t\,\rd W_t,$$
we have
$$\pt \Svv_t = 2\gamma\beta^{-1} - 2\Svv_t - 2\Sxv_t.$$
Combining the ODEs together, we get
\begin{equation}\label{eq:ULD_Sigma_ODE}
\pt \begin{bmatrix} \Sxx_t \\ \Sxv_t \\ \Svv_t \end{bmatrix} = \begin{bmatrix} 0&2&0\\ -1&-\gamma&1\\ 0&-2&-2\gamma \end{bmatrix} \begin{bmatrix} \Sxx_t \\ \Sxv_t \\ \Svv_t \end{bmatrix} + \begin{bmatrix} 0\\0\\2\gamma\beta^{-1} \end{bmatrix}.
\end{equation}
In practice, we choose $\beta=\gamma=1$ and approximate the solution to $\Sigma_t$ using Runge–-Kutta $4$ (RK4) method. Figure \ref{fig:ULDGaussian_Sigma} shows the reference covariance evolution obtained from the RK4 scheme, together with the empirical estimation of the covariance with simulated $\num{e4}$ trajectories. We observe that the reference covariance coincides with its empirical estimation. Additionally, the empirical curve for $\Sxx_t$ is smoother because $x_t$ does not have noise.
\begin{figure}[!htbp]
\centering
\includegraphics[width=\linewidth]{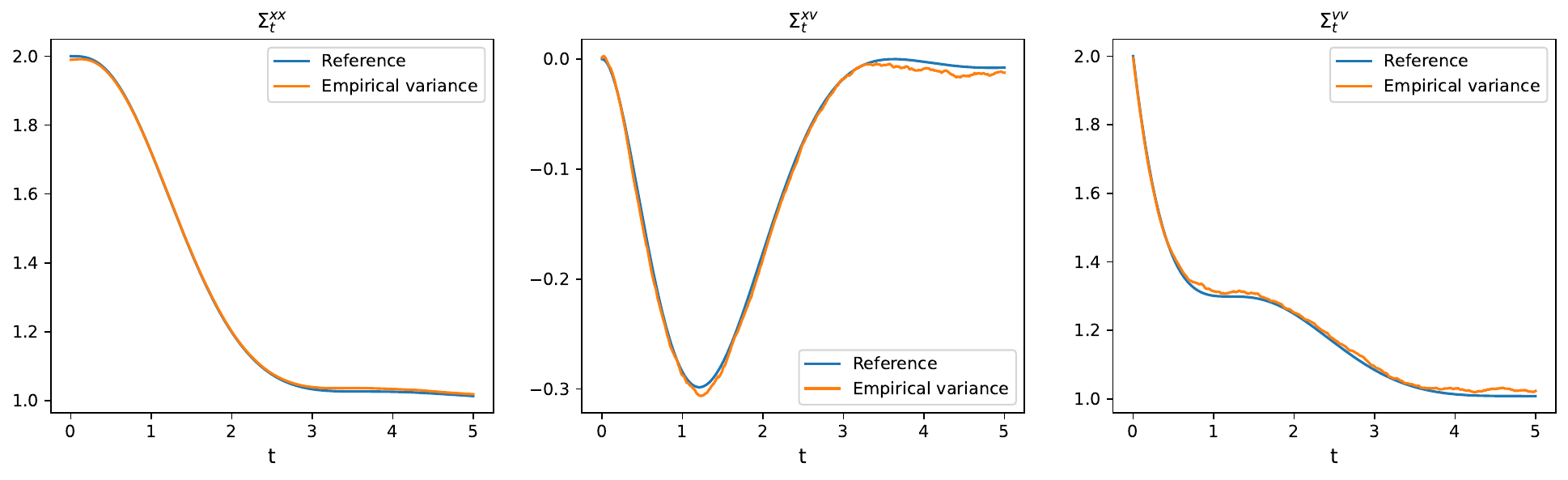}
\caption{Reference covariance evolution computed from RK4 scheme and empirical estimation.}
\label{fig:ULDGaussian_Sigma}
\end{figure}

\subsection{Score-based normalizing flow for Underdamped Langevin dynamics}
In this subsection, we derive the score-based normalizing flow for ULD. We denote the composed velocity for ULD as
$$f(t,x,v) = \begin{bmatrix} 
v \\ f_v(t,x,v) \end{bmatrix}.$$
Then, the probability flow dynamic for the state is 
$$\pt \begin{bmatrix} x_t\\v_t \end{bmatrix} = \begin{bmatrix} v_t \\ f_v(t,x_t,v_t) \end{bmatrix}.$$
By the ODE dynamics in Proposition \ref{prop:ODE}, $L_t=\rho(t,x_t,v_t)$, $l_t=\log\rho(t,x_t,v_t)$, $s^x_t = \nx \log \rho(t,x_t,v_t)$, and $s^v_t = \nv \log \rho(t,x_t,v_t)$ satisfy
\begin{equation}\label{eq:ULD_ODEs}
\begin{aligned}
\pt L_t &= - \nv \cdot f_v(t,x_t,v_t) L_t, \\
\pt l_t &= - \nv \cdot f_v(t,x_t,v_t), \\
\pt \begin{bmatrix} s^x_t\\s^v_t \end{bmatrix} &= -\begin{bmatrix} 0&I_d\\\nx f_v & \nv f_v \end{bmatrix}\tp \begin{bmatrix} s^x_t\\s^v_t \end{bmatrix} - \begin{bmatrix} \nx(\nv\cdot f_v) \\ \nv(\nv\cdot f_v) \end{bmatrix}.
\end{aligned}
\end{equation}

In Gaussian case where $\begin{bmatrix}x_t\\v_t\end{bmatrix} \sim N(0,\Sigma_t)$, the density function is
$$\rho(t,x,v) = (2\pi)^{-d} \det(\Sigma_t)^{-\frac12} \exp\parentheses{ -\frac12 \begin{bmatrix}x\\v\end{bmatrix}\tp \Sigma_t^{-1} \begin{bmatrix}x\\v\end{bmatrix}},$$
where the covariance matrix $\Sigma_t$ satisfies \eqref{eq:ULD_Sigma_ODE}. Consequently, a reference solution to the composed velocity is
$$f_v(t,x,v) = -(x+\gamma v) + \dfrac{\gamma}{\beta}\parentheses{(\Sigma_t^{-1})^{xv} x + (\Sigma_t^{-1})^{vv} v}.$$
We can also get the reference score function $s_t = -\Sigma_t^{-1} \begin{bmatrix}x_t\\v_t\end{bmatrix}$. The reference free energy and its dissipation can be computed through \eqref{eq:ULD_energy} and \eqref{eq:ULD_dissipation}.

\section{Details for numerical implementations}
We present the details for numerical implementation in this section.

\subsection{Details for Underdamped Langevin dynamic}
In the ULD example, the numerical discretization for the state and other dynamics in \eqref{eq:ULD_ODEs} are
\begin{align*}
x_{t_{j+1}} &= x_{t_j} + \dt \, v_{t_j},\\
v_{t_{j+1}} &= v_{t_j} + \dt \, \fvt(t_j, x_{t_j},v_{t_j}),\\
L_{t_{j+1}} &= L_{t_j} - \dt \, \nv\cdot \fvt(t_j, x_{t_j},v_{t_j}) L_{t_j},\\
l_{t_{j+1}} &= l_{t_j} - \dt \, \nv\cdot \fvt(t_j, x_{t_j},v_{t_j}),\\
s^x_{t_{j+1}} &= s^x_{t_j} - \dt \parentheses{ \nx \fvt(t_j, x_{t_j},v_{t_j})\tp s^v_{t_j} + \nx( \nv\cdot \fvt(t_j, x_{t_j},v_{t_j})) },\\
s^v_{t_{j+1}} &= s^v_{t_j} - \dt \parentheses{s^x_{t_j}+ \nv \fvt(t_j, x_{t_j},v_{t_j})\tp s^v_{t_j} + \nv( \nv\cdot \fvt(t_j, x_{t_j},v_{t_j})) }.
\end{align*}
As an analog to \eqref{eq:free_energy_discrete} and \eqref{eq:energy_dissipation_discrete}. An estimation of the free energy \eqref{eq:ULD_energy} and its dissipation \eqref{eq:ULD_dissipation} for ULD is obtained through
$$\widehat{\Dh}(\rho(t_j,\cdot,\cdot)) = \dfrac{1}{N_x} \sum_{n=1}^{N_x} \parentheses{l^{(n)}_{t_j} + \beta H(x^{(n)}_{t_j}, v^{(n)}_{t_j})},$$
and 
$$\widehat{\pt \Dh}(\rho(t_j,\cdot,\cdot)) = - \dfrac{\gamma}{N_x \beta} \sum_{n=1}^{N_x} \abs{s^{v\,(n)}_{t_j} - \nv\log \pi(x^{(n)}_{t_j}, v^{(n)}_{t_j})}^2.$$

For ULD with a quadratic potential $U(x)=\frac12|x|^2$, the state dynamic demonstrates a clear Hamiltonian structure. In this setting, we apply a symplectic scheme to enhance the stability, where we compute $x_{t_{j+1}} = x_{t_j} + \dt \, v_{t_{j+1}}$ after getting $v_{t_{j+1}}$. The update rules for other quantities remains unchanged. For general potential function, we keep using the forward Euler scheme. We conclude the numerical algorithm for ULD with quadratic and general potential in Algorithm \ref{alg:ULDGaussian} and \ref{alg:ULDDW}. 
For both cases, we split intervals to overcome the issue of long time horizon. We train a neural network $f_m:=f_{v \theta_m}$ within each stage as introduced in \ref{sec:algorithm_long}.

\begin{algorithm}[htbp!]
\caption{Flow matching solver for ULD with quadratic potential}
\begin{algorithmic}[1]
\STATE \textbf{Input:}~ Parameter $\beta, \gamma$, neural network structure for each stage, number of steps for the first stage $N_{\text{step}0}$ and the following stages $N_{\text{step}}$, learning rages, $N_x,N_t, N_T$
\STATE \textbf{Output:}~ Approximated composed velocity fields $\{f_m(\cdot,\cdot)\}_{m=1}^{N_T}$ for the MFC problem
\STATE \textbf{Initialization:}~Parameter $\theta_1$ for the first neural network $f_1(\cdot,\cdot)$, $\dt=T/(N_T N_t)$
\FOR{$\text{step}=1,2,\ldots, N_{\text{step}0}$}
\STATE{$\text{Loss} = 0$, $t_0=0$, Sample $\{(x^{(n)}_0, v^{(n)}_0)\}_{n=0}^{N_x}$ i.i.d. from $\rho_0$}
\STATE{Compute $s_0^{x \, (n)} = \nx \log\rho_0(x^{(n)}_0, v^{(n)}_0)$ and $s_0^{v\,(n)} = \nv \log\rho_0(x^{(n)}_0, v^{(n)}_0)$}
\FOR{$j=0,1,\ldots,N_t-1$}
\STATE{$\text{Loss} = \text{Loss} + \frac{1}{N_x} \sum_{n=1}^{N_x} |f_1(t_j,x^{(n)}_{t_j},v^{(n)}_{t_j}) + (\gamma v^{(n)}_{t_j} +x^{(n)}_{t_j}) + \gamma\beta^{-1} {s_{t_j}^v}^{(n)}|^2\dt$}
\STATE{$t_{j+1} = t_0 + (j+1)\dt$}\hfill\COMMENT{set time stamp}
\STATE{$v^{(n)}_{t_{j+1}} = v^{(n)}_{t_j} + \dt \, f_1(t_j, x^{(n)}_{t_j},v^{(n)}_{t_j})$,~ $x^{(n)}_{t_{j+1}} = x^{(n)}_{t_j} + \dt \, v^{(n)}_{t_{j+1}}$}\hfill\COMMENT{update states}
\STATE{$s_{t_{j+1}}^{x\,(n)} = s^{x\,(n)}_{t_j} - \dt\sqbra{ \nx f_1(t_j, x^{(n)}_{t_j}, v^{(n)}_{t_j})\tp s^{v\,(n)}_{t_j} + \nx( \nv\cdot f_1(t_j, x^{(n)}_{t_j},v^{(n)}_{t_j})) }$}
\STATE{$s_{t_{j+1}}^{v\,(n)} = s^{v\,(n)}_{t_j} - \dt \sqbra{s^{x\,(n)}_{t_j} + \nv f_1(t_j, x^{(n)}_{t_j}, v^{(n)}_{t_j})\tp s^{v\,(n)}_{t_j} + \nx( \nv\cdot f_1(t_j, x^{(n)}_{t_j},v^{(n)}_{t_j})) }$}
\ENDFOR
\STATE{Update $\theta_m$ through Adam method to minimize $\text{Loss}$}
\ENDFOR
\FOR{$m = 2,3,\ldots,N_T$} 
\STATE{Initialize $\theta_m$ as $\theta_{m-1}$ from previous stage}
\hfill \COMMENT{warm-start}
\FOR{$\text{step}=1,2,\ldots, N_{\text{step}}$}
\STATE{Sample $\{(x^{(n)}_0, v^{(n)}_0)\}_{n=0}^{N_x}$ i.i.d. from $\rho_0$}
\STATE{Compute $s_0^{x \, (n)} = \nx \log\rho_0(x^{(n)}_0, v^{(n)}_0)$ and $s_0^{v\,(n)} = \nv \log\rho_0(x^{(n)}_0, v^{(n)}_0)$}
\FOR{$m' = 1,\ldots,m-1$}
\STATE{$t_0 = (m'-1)T'$}\hfill\COMMENT{initial time for the stage}
\FOR{$j=0,\ldots,N_t-1$}
\STATE{$t_{j+1} = ((m'-1)N_t+j+1)\dt$}\hfill\COMMENT{set time stamp}
\STATE{$v^{(n)}_{t_{j+1}} = v^{(n)}_{t_j} + \dt \, f_{m'}(t_j, x^{(n)}_{t_j},v^{(n)}_{t_j})$,~ $x^{(n)}_{t_{j+1}} = x^{(n)}_{t_j} + \dt \, v^{(n+1)}_{t_{j+1}}$}\hfill\COMMENT{update states}
\STATE{$s_{t_{j+1}}^{x\,(n)} = s^{x\,(n)}_{t_j} - \dt\sqbra{ \nx f_{m'}(t_j, x^{(n)}_{t_j}, v^{(n)}_{t_j})\tp s^{v\,(n)}_{t_j} + \nx( \nv\cdot f_{m'}(t_j, x^{(n)}_{t_j},v^{(n)}_{t_j})) }$}
\STATE{$s_{t_{j+1}}^{v\,(n)} = s^{v\,(n)}_{t_j} - \dt \sqbra{s^{x\,(n)}_{t_j} + \nv f_{m'}(t_j, x^{(n)}_{t_j}, v^{(n)}_{t_j})\tp s^{v\,(n)}_{t_j} + \nx( \nv\cdot f_{m'}(t_j, x^{(n)}_{t_j},v^{(n)}_{t_j})) }$}
\ENDFOR
\ENDFOR
\STATE{$\text{Loss} = 0$, $t_0 = (m-1)T'$}
\STATE{Follow steps in line 7-14 to train $f_m$ with initial data $\{x_{t_0}^{(n)}, v_{t_0}^{(n)}, s_{t_0}^{x\,(n)}, s_{t_0}^{v\,(n)}\}_{n=1}^{N_x}$}
\ENDFOR
\ENDFOR
\end{algorithmic}
\label{alg:ULDGaussian}
\end{algorithm}

\begin{algorithm}[htbp!]
\caption{Flow matching solver for ULD with general potential}
\begin{algorithmic}[1]
\STATE \textbf{Input:}~ Parameter $\beta, \gamma$, neural network structure for each stage, number of steps for the first stage $N_{\text{step}0}$ and the following stages $N_{\text{step}}$, learning rages, $N_x,N_t, N_T$
\STATE \textbf{Output:}~ Approximated composed velocity fields $\{f_m(\cdot,\cdot)\}_{m=1}^{N_T}$ for the MFC problem
\STATE \textbf{Initialization:}~Parameter $\theta_1$ for the first neural network $f_1(\cdot,\cdot)$, $\dt=T/(N_T N_t)$
\FOR{$\text{step}=1,2,\ldots, N_{\text{step}0}$}
\STATE{$\text{Loss} = 0$, $t_0=0$, Sample $\{(x^{(n)}_0, v^{(n)}_0)\}_{n=0}^{N_x}$ i.i.d. from $\rho_0$}
\STATE{Compute $s_0^{x \, (n)} = \nx \log\rho_0(x^{(n)}_0, v^{(n)}_0)$ and $s_0^{v\,(n)} = \nv \log\rho_0(x^{(n)}_0, v^{(n)}_0)$}
\FOR{$j=0,1,\ldots,N_t-1$}
\STATE{$\text{Loss} = \text{Loss} + \frac{1}{N_x} \sum_{n=1}^{N_x} |f_1(t_j,x^{(n)}_{t_j},v^{(n)}_{t_j}) + (\gamma v^{(n)}_{t_j} + \nx U(x^{(n)}_{t_j})) + \gamma\beta^{-1} {s_{t_j}^v}^{(n)}|^2\dt$}
\STATE{$t_{j+1} = t_0+ (j+1)\dt$}\hfill\COMMENT{set time stamp}
\STATE{$v^{(n)}_{t_{j+1}} = v^{(n)}_{t_j} + \dt \, f_1(t_j, x^{(n)}_{t_j},v^{(n)}_{t_j})$,~ $x^{(n)}_{t_{j+1}} = x^{(n)}_{t_j} + \dt \, v^{(n)}_{t_j}$}\hfill\COMMENT{update states}
\STATE{$s_{t_{j+1}}^{x\,(n)} = s^{x\,(n)}_{t_j} - \dt\sqbra{ \nx f_1(t_j, x^{(n)}_{t_j}, v^{(n)}_{t_j})\tp s^{v\,(n)}_{t_j} + \nx( \nv\cdot f_1(t_j, x^{(n)}_{t_j},v^{(n)}_{t_j})) }$}
\STATE{$s_{t_{j+1}}^{v\,(n)} = s^{v\,(n)}_{t_j} - \dt \sqbra{s^{x\,(n)}_{t_j} + \nv f_1(t_j, x^{(n)}_{t_j}, v^{(n)}_{t_j})\tp s^{v\,(n)}_{t_j} + \nx( \nv\cdot f_1(t_j, x^{(n)}_{t_j},v^{(n)}_{t_j})) }$}
\ENDFOR
\STATE{Update $\theta_m$ through Adam method to minimize $\text{Loss}$}
\ENDFOR
\FOR{$m = 2,3,\ldots,N_T$} 
\STATE{Initialize $\theta_m$ as $\theta_{m-1}$ from previous stage}
\hfill \COMMENT{warm-start}
\FOR{$\text{step}=1,2,\ldots, N_{\text{step}}$}
\STATE{Sample $\{(x^{(n)}_0, v^{(n)}_0)\}_{n=0}^{N_x}$ i.i.d. from $\rho_0$}
\STATE{Compute $s_0^{x \, (n)} = \nx \log\rho_0(x^{(n)}_0, v^{(n)}_0)$ and $s_0^{v\,(n)} = \nv \log\rho_0(x^{(n)}_0, v^{(n)}_0)$}
\FOR{$m' = 1,\ldots,m-1$}
\STATE{$t_0 = (m'-1)T'$}\hfill\COMMENT{initial time for the stage}
\FOR{$j=0,\ldots,N_t-1$}
\STATE{$t_{j+1} = ((m'-1)N_t+j+1)\dt$}\hfill\COMMENT{set time stamp}
\STATE{$v^{(n)}_{t_{j+1}} = v^{(n)}_{t_j} + \dt \, f_{m'}(t_j, x^{(n)}_{t_j},v^{(n)}_{t_j})$,~ $x^{(n)}_{t_{j+1}} = x^{(n)}_{t_j} + \dt \, v^{(n)}_{t_{j+1}}$}\hfill\COMMENT{update states}
\STATE{$s_{t_{j+1}}^{x\,(n)} = s^{x\,(n)}_{t_j} - \dt\sqbra{ \nx f_{m'}(t_j, x^{(n)}_{t_j}, v^{(n)}_{t_j})\tp s^{v\,(n)}_{t_j} + \nx( \nv\cdot f_{m'}(t_j, x^{(n)}_{t_j},v^{(n)}_{t_j})) }$}
\STATE{$s_{t_{j+1}}^{v\,(n)} = s^{v\,(n)}_{t_j} - \dt \sqbra{s^{x\,(n)}_{t_j} + \nv f_{m'}(t_j, x^{(n)}_{t_j}, v^{(n)}_{t_j})\tp s^{v\,(n)}_{t_j} + \nx( \nv\cdot f_{m'}(t_j, x^{(n)}_{t_j},v^{(n)}_{t_j})) }$}
\ENDFOR
\ENDFOR
\STATE{$\text{Loss} = 0$, $t_0 = (m-1)T'$}
\STATE{Follow steps in line 7-14 to train $f_m$ with initial data $\{x_{t_0}^{(n)}, v_{t_0}^{(n)}, s_{t_0}^{x\,(n)}, s_{t_0}^{v\,(n)}\}_{n=1}^{N_x}$}
\ENDFOR
\ENDFOR
\end{algorithmic}
\label{alg:ULDDW}
\end{algorithm}

\subsection{Hyperparameters and other details}
In this subsection, we present the hyperparameters and other details for the numerical experiments.

\subsubsection*{Neural network} We parametrize the composed velocity field as the known drift field $b(x)$ added by a multilayer perceptron with $\tanh$ as the activation function. The drift field serve as a baseline that will benefit the training procedure. For all the experiments in $2$ and $3$ dimensions, we apply $2$-layer networks with $100$ neurons in the hidden layers. For $50$ and $100$ dimensional examples, the widths of the neural networks are $400$.

\subsubsection*{Hyperparameters}
For all the experiments, we apply a step size $\dt=0.01$ and learning rate of $0.01$. The other hyperparameters are presented in Table \ref{tab:hyperparameters}.
\begin{table}[htbp!]
\centering
\begin{tabular}{c|ccccc}
\midrule
Example & $(N_{\text{step}0}, N_{\text{step}})$ & $(T,N_T)$ & $N_x$ & other parameters\\
\midrule
Langevin OU & $(,500)$ & $(1.0,1)$ & $500$ & $\ve=0.5$, $c=0.5$ \\
Langevin double-well & $(,500)$ & $(1.0,1)$ & $500$ & $\ve=0.5$, $c=0.5$ \\
ULD Gaussian & $(500,200)$ & $(5.0,5)$ & $1000$ & $\Sigma_0=2I_2$, $\beta=\gamma=1$ \\
ULD double-well & $(500,200)$ & $(5.0,25)$ & $1000$ & $\Sigma_0=I_2$, $\beta=\gamma=1$ \\
Lorenz & $(500,200)$ & $(5.0,25)$ & $1000$ & $\ve=0.1$, $s=0.2$\\
Arctangent Lorenz & $(500,200)$ & $(5.0,25)$ & $1000$ & $\ve=0.1$, $s=0.1$\\
Van der Pol & $(500,200)$ & $(2.0,10)$ & $1000$ & $\ve=0.1$, $\mu=2.0$\\
Swimmer & $(500,200)$ & $(5.0,25)$ & $1000$ & $\ve=1.0$, $\gamma=0.1$\\
Langevin OU ($50d$) & $(,5000)$ & $(1.0,1)$ & $1000$ & $\ve=0.5$, $c=0.5$ \\
Langevin double-well ($50d$) & $(,5000)$ & $(1.0,1)$ & $1000$ & $\ve=0.5$, $c=0.5$ \\
ULD Gaussian ($50d$) & $(2000,1000)$ & $(5.0,5)$ & $1000$ & $\Sigma_0=2I_{50}$, $\beta=\gamma=1$ \\
ULD double-well ($50d$) & $(2000,1000)$ & $(5.0,25)$ & $1000$ & $\Sigma_0=I_{50}$, $\beta=\gamma=1$ \\
Harmonic particles ($100d$) & $(2000,1000)$ & $(6.0,6)$ & $400$ & $a=2,\omega=1,\alpha=0.5,\varepsilon=0.25$\\
\midrule
\end{tabular}
\caption{Hyperparameters for the numerical examples. $c$ is the coefficient for anti-symmetry part of the drift function. $N_T$ is the number of subintervals that we partition into. $\ve$ is the noise level. $s$ is the scaling parameter introduced in Section \ref{sec:chaotic}.}
\label{tab:hyperparameters}
\end{table}

\section*{Declaration of generative AI and AI-assisted technologies in the writing process}
During the preparation of this work the authors used ChatGPT and Grammarly in order to polish the language of the article and correct grammatical mistakes. After using these tools, the authors reviewed and edited the content as needed and take full responsibility for the content of the published article.

\medskip
\noindent\textbf{Acknowledgements}. 
M. Zhou's work is partially supported by the AFOSR YIP award No. FA9550-23-1-0087. W. Li's work is supported by the AFOSR YIP award No. FA9550-23-1-0087, NSF RTG: 2038080, and NSF DMS: 2245097.

\bibliographystyle{elsarticle-num} 
\bibliography{ref}

\end{document}